\documentclass[12pt,reqno]{amsart}
\usepackage{amssymb}
\usepackage{txfonts}
\usepackage{amsfonts}
\usepackage{amsmath}
\usepackage{graphicx}
\usepackage{hyperref}
\usepackage{float}
\usepackage{epstopdf}
\usepackage{color}

\allowdisplaybreaks
\newtheorem{theorem}{Theorem}[section]
\newtheorem{proposition}[theorem]{Proposition}
\newtheorem{lemma}[theorem]{Lemma}

\newtheorem{definition}[theorem]{Definition}

\renewcommand{\theequation}{\thesection.\arabic{equation}}

\input{tcilatex}

\def\oint{\fint}

\def\FRAME#1#2#3#4#5#6#7#8
{
 \begin{figure}[ptbh]
 \begin{center}
 \includegraphics[height=#3]{#7}
 \caption{#5}
 \label{#6}
 \end{center}
 \end{figure}
}

\begin{document}

\title{BGD domains in p.c.f. self-similar sets I: boundary value problems for harmonic functions}

\pagestyle{plain}

\author {Qingsong Gu}
\address{School of Mathematics\\ Nanjing University, Nanjing, 210093, China}
\email{qingsonggu@nju.edu.cn}
{\author {Hua Qiu}}
\address {School of Mathematics\\ Nanjing University, Nanjing, 210093, China}
\email {huaqiu@nju.edu.cn}

\thanks {The first author was supported by National Natural Science Foundation of China (Grant Nos. 12101303 and 12171354); the second author was support by National Natural Science Foundation of China (Grant No. 12471087).}

\maketitle
\begin{abstract}  {We study the boundary value problems for harmonic functions on open connected subsets of post-critically finite (p.c.f.) self-similar sets, on which the Laplacian is defined through a strongly recurrent self-similar local regular Dirichlet form. For a p.c.f. self-similar set $K$, we prove that for any  open connected subset $\Omega\subset K$ whose ``geometric'' boundary  is a graph-directed self-similar set, there exists a finite number of matrices called \textit{flux transfer matrices} whose products generate the hitting probability from a point in $\Omega$ to the ``resistance'' boundary $\partial \Omega$. The harmonic functions on $\Omega$ can be expressed by integrating functions on $\partial \Omega$ against the probability measures. Furthermore, we obtain a two-sided estimate of the energy of a harmonic function  in terms of its values on  $\partial \Omega$.}
\end{abstract}

\maketitle
\section{\bf Introduction}\label{sec1}
\setcounter{equation}{0}\setcounter{theorem}{0}

Let $\Omega$ be a smooth domain in $\mathbb R^n$ and $\Delta=\sum_{i=1}^{n}\frac{\partial^2}{\partial x_i^2}$ be the Laplace operator. It is known that the Dirichlet problem
\begin{equation}
\left\{
  \begin{array}{ll}
    \Delta u=0  & \hbox{in $\Omega$}, \\
    u=f  & \hbox{on $\partial \Omega$},
  \end{array}
\right.
\end{equation}
has a unique solution $u$ for any continuous function $f$ on the boundary. In particular, if $\Omega$ is the open unit ball $B=\{x\in\mathbb R^n:\ |x|<1\}$, $u$ has an expression as the Poisson integral
\begin{equation}
u(x)=\int_{|y|=1}f(y)P(x,y)d\sigma(y),
\end{equation}
where $d\sigma$ is the normalized surface measure on the unit sphere and $P(x,y)=\frac{1-|x|^2}{|x-y|^{n}}$ is the Poisson kernel. From the probabilistic point of view, the measure $P(x,y)d\sigma(y)$ represents the hitting probability of the Brownian motion from $x$ in $B$ to the sphere.

On fractals,  a local regular Dirichlet form plays the role of the Dirichlet integral $\int_{\Omega} |\nabla u|^2dx$ in a domain $\Omega$ of $\mathbb R^n$, and it has an infinitesimal generator $\Delta$ called  the Laplacian. The construction of Dirichlet forms on fractals is motivated by the study of Brownian motions on self-similar sets in a probabilistic approach, with pioneering works of Kusuoka \cite{Kus}, Goldstein \cite{G} and Barlow-Perkins \cite{BP} on the Sierpinski gasket and of Lindstr\o m \cite{L} on nested fractals, and also of Barlow-Bass \cite{BB} and Kusuoka-Zhou \cite{KZ} on the Sierpinski carpet.
There is also a large literature on the topic based on Kigami's analytic approach on the post-critically finite (p.c.f.) self-similar sets (see \cite{B, FS, HMT, K1, K2, K,Me,Pe,Sa, S} and the references therein).

Specifically, let $K$ be a self-similar set generated by an iterated function system $\{F_i\}_{i=1}^N$ on a complete metric space.
Most of the previous studies are about the Dirichlet forms $(\mathcal E,\mathcal F)$ satisfying the energy self-similar identity, which means that there exist $N$ positive real numbers $\{ r_i\}_{i=1}^{N}$ called {\it energy renormalizing factors} such that for any function $u\in\mathcal F$, it holds that $u\circ F_i\in\mathcal F$ for any $i=1,\ldots,N$, and
\begin{equation*}
\mathcal E[u]=\sum_{i=1}^N\frac1{r_i}\mathcal E[u\circ F_i],
\end{equation*}
where $\mathcal{E}[u]:=\mathcal{E}(u,u)$. If further $r_i\in(0,1)$ for each $1\leq i\leq N$, then the form is strongly recurrent. Such forms are known to exist on some classes of self-similar sets having certain symmetry properties, for example, nested fractals \cite{L, Sa}, affine nested fractals \cite{FHK}, and Sierpinski carpets \cite{BB,KZ}.

\bigskip

For a given p.c.f. self-similar set $K$ equipped with a strongly recurrent self-similar Dirichlet form, we are concerned with the boundary value problems for harmonic functions on a domain $\Omega$ in $K$ (which means $\Omega$ is a nonempty open connected subset of $K$). We mainly focus on two problems originated from classical analysis: one is to find the exact description of the hitting probability from a point in $\Omega$ to the boundary; the other is to estimate the energy of a harmonic function generated by its boundary values. From the analytic point of view, we should regard $\Omega$ as a resistance space; see the work of Kigami and Takahashi \cite{KT22} on a particular $\Omega$, the Sierpinski gasket (SG) minus its bottom line. This leads us to introduce the topology given by the resistance metric to replace the underlying topology inherited from $K$. So in our investigation, we need to discriminate between two different boundaries of $\Omega$. We call them ``resistance" boundary and ``geometric" boundary later.


\begin{figure}[h]
	\includegraphics[width=7cm]{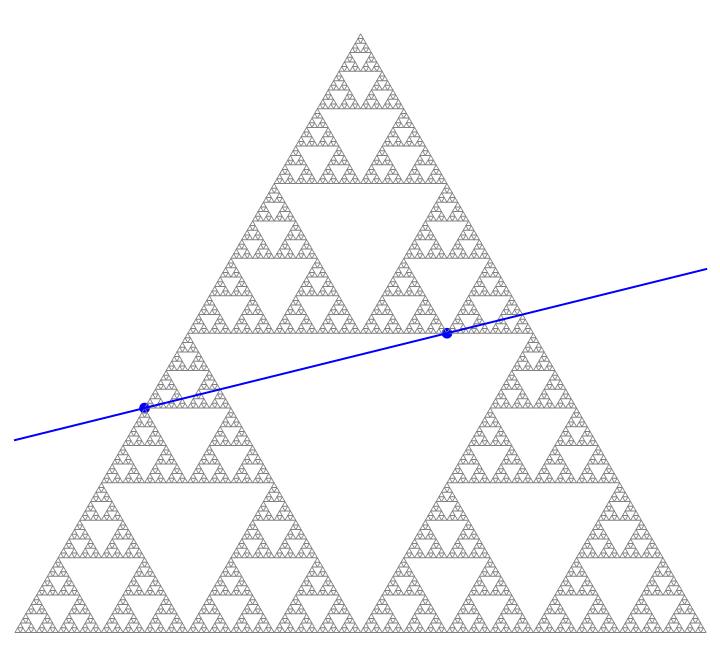}
	\begin{picture}(0,0)
	\end{picture}	
	\caption{domains in the Sierpinski gasket} \label{figureSGcut}
\end{figure}

The study of such problems was initiated in \cite{OS,GKQS,KT22} for typical domains in SG, see also \cite{LS,CQ} for extensions in more generalized SGs. However, the techniques strongly depend on the specific structure of SG and the geometric structure of the domain. For a general p.c.f. self-similar set $K$, due to its  self-similarity, it is natural to consider domains whose geometric boundaries are {\it graph-directed self-similar sets}, for example, domains in SG generated by ``cutting" with an oblique line (see Figure \ref{figureSGcut} and Subsection \ref{subsec7.1}).  Another example is a family of domains in Lindstr\o m's snowflake whose boundaries are Koch curves (see Figure \ref{figure1} with boundaries drawn in thick lines).


\begin{figure}[h]
	\includegraphics[width=4.5cm]{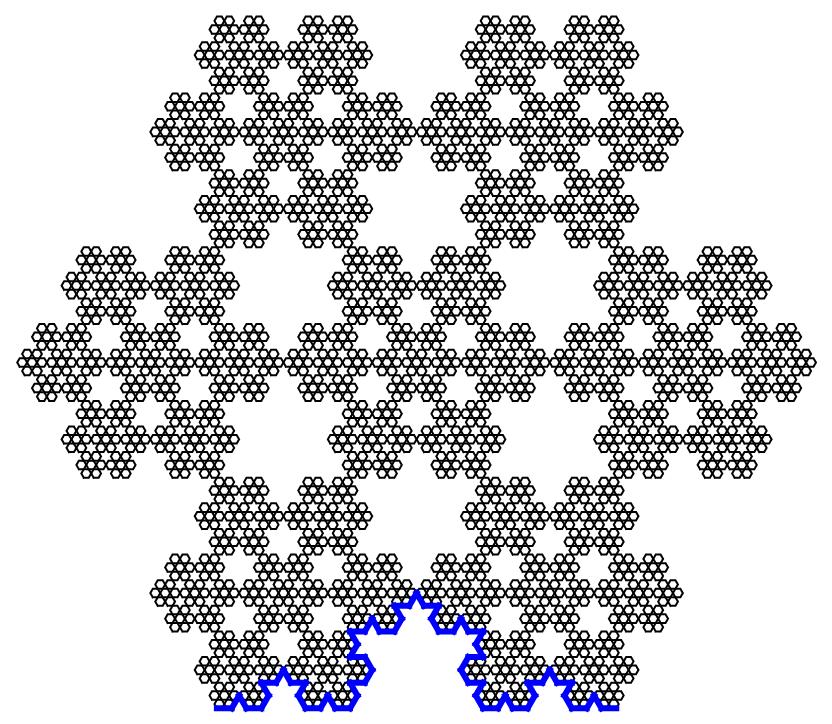}
			\includegraphics[width=4.5cm]{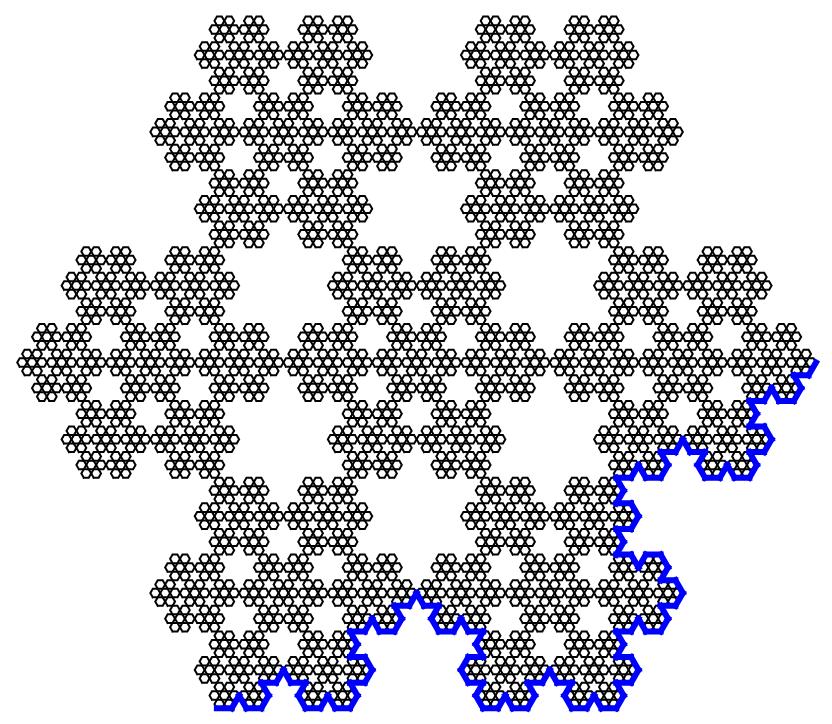}
				\includegraphics[width=4.5cm]{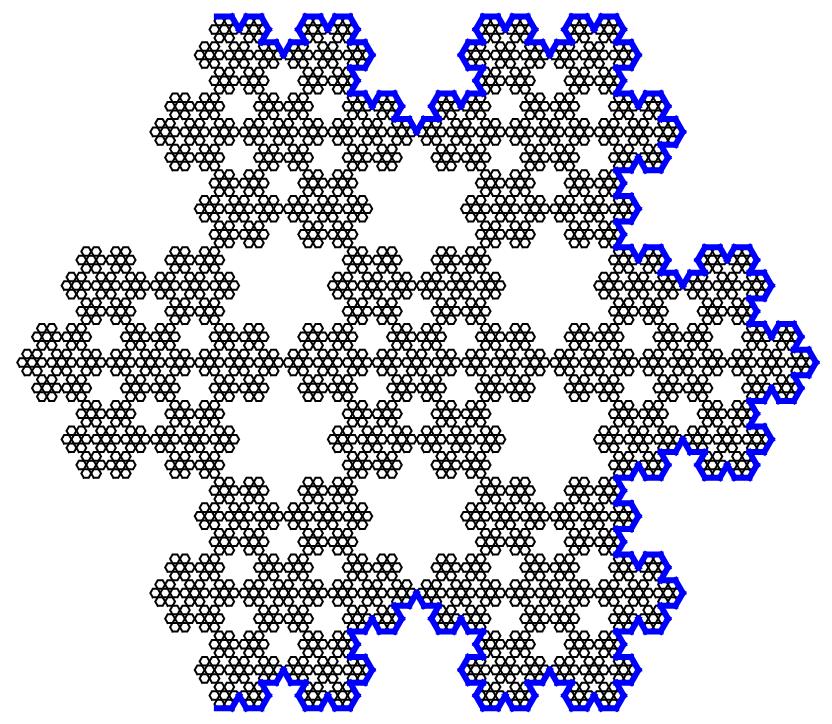}
	\caption{Domains in Lindstr\o m's snowflake} \label{figure1}
\end{figure}

\medskip

In this paper, we propose a condition called \textit{boundary graph-directed condition} (BGD for short) for a finite collection of domains $\Omega_i$, $1\leq i\leq P$ with geometric boundaries $D_i$:

\medskip

\noindent(BGD):\quad \textit{for $1\leq i\leq P$ and $1\leq k\leq N$, if $\Omega_i\cap F_k(K)\neq\emptyset$ and $D_i\cap F_k(K)\neq\emptyset$, then there exists $1\leq j\leq P$ such that
\begin{equation}\label{graphdirected}
\Omega_i\cap F_k(K)=F_k(\Omega_j),\qquad D_i\cap F_k(K)=F_k(D_j).
\end{equation}}

Under this condition, we are able to solve the boundary value problems for domains in general p.c.f. self-similar sets.

Firstly, to determine the hitting probability from a  point in a domain $\Omega$ to its resistance boundary $\partial \Omega$, we introduce a finite number of matrices, called {\it flux transfer matrices}, and prove that the products of these matrices generate the hitting probability (see Theorem \ref{th3.1}). We note that these matrices are determined not only by the resistance form on the fractal but also by the graph-directed structure of the domain.

Secondly, using the hitting probability measures, we establish an equivalent characterization of energies of harmonic functions through their boundary values (see Theorem \ref{enes}). 
We remark that a closely related problem is to consider the trace of functions with finite energy on a self-similar set to its subsets. In \cite{HiK}, Hino and Kumagai proved a trace theorem for these functions on a self-similar set to its self-similar subsets, extending the result of Jonsson \cite{J} for the trace of functions on SG to its bottom line.
\medskip

We organize the paper as follows. In Section \ref{sec2}, we give some preliminaries for strongly recurrent self-similar Dirichlet forms on p.c.f. self-similar sets and recall some basic properties of electric networks. In Section \ref{sec3}, we give several basic properties of the boundary graph-directed (BGD) condition to describe the geometric boundary of a domain in self-similar sets. In Section \ref{sec4}, for a domain satisfying BGD, we study its resistance boundary in terms of resistance forms and characterize this boundary as a symbolic space.
In Section \ref{sec5}, by introducing the flux transfer matrices, we prove Theorem \ref{th3.1} on the expression of hitting probabilities. In Section \ref{sec6}, we prove Theorem \ref{enes} on the energy estimate of harmonic functions in terms of their boundary values. Finally in Section \ref{sec7}, we present several examples.

Throughout the paper, we use the notation $f\lesssim(\gtrsim) g$ for two variables $f$ and $g$ if there exists a constant $C>0$ such that $f\leq(\geq) C g$, and also $f\asymp g$ if both $f\lesssim g$ and $f\gtrsim g$ hold. For a set $A$, we write $\ell(A)$ for the collection of real valued functions on $A$.

In a subsequent paper \cite{GQ}, we investigate the Weyl's law for the asymptotics of the eigenvalues of Laplacians on BGD domains. Especially, we mainly focus on the second-order remainder estimates which depends on the geometry of the boundaries.   

\section{\bf Preliminaries}\label{sec2}
\setcounter{equation}{0}\setcounter{theorem}{0}

We first recall some notations about post-critically finite (p.c.f. for short) self-similar sets introduced by Kigami \cite{K2,K}. Let $N\geq 2$ be an integer, $\{F_i\}_{i=1}^N$ be a finite collection of injective contractions on a complete metric space $(X,d)$. The self-similar set associated with the iterated function system (IFS) $\{F_i\}_{i=1}^N$ is the unique nonempty compact set $K\subset X$ satisfying $$K=\bigcup_{i=1}^N F_i(K).$$ We define the symbolic space as usual. Let $\Sigma=\{1,\ldots,N\}$ be the {\it alphabets}, $\Sigma^n$ the set of {\it words} of length $n$ (where $\Sigma^0=\{\emptyset\}$ containing only the empty word), and $\Sigma^\infty$ the set of {\it infinite words} $\omega=\omega_1\omega_2\cdots$. For $\omega=\omega_1\cdots\omega_n\in\Sigma^n$, we write $|\omega|=n$, $F_{\omega}=F_{\omega_1}\circ \cdots\circ F_{\omega_n}$ and call $F_\omega(K)$ an {\it $n$-cell}  ($F_\emptyset=\rm{Id}$). Let $\pi: \Sigma^\infty\rightarrow K$ be defined by $\{x\}=\{\pi(\omega)\}=\bigcap_{n\geq1}F_{[\omega]_n}(K)$ with $[\omega]_n=\omega_1\cdots\omega_n$, a symbolic representation of $x\in K$ by $\omega$.

Following \cite{K}, we define the critical set $\mathcal C$ and post-critical set $\mathcal P$ for $K$ by
\begin{equation*}
\mathcal C=\pi^{-1}\left(\bigcup_{1\leq i<j\leq N}\left(F_i(K)\cap F_j(K)\right)\right),\qquad \mathcal P=\bigcup_{m\geq1}\sigma^m(\mathcal C),
\end{equation*}
where  $\sigma:\ \Sigma^\infty\rightarrow\Sigma^\infty$ is the left shift operator, i.e. $\sigma(\omega_1\omega_2\cdots)=\omega_2\omega_3\cdots$. If $\mathcal P$ is finite, we call $\{F_i\}_{i=1}^N$ a {\it post-critically finite (p.c.f.)} IFS, and $K$ a {\it p.c.f. self-similar set}. The boundary of $K$ is defined by $V_0=\pi(\mathcal P)$. We also inductively denote
\begin{equation*}
V_n=\bigcup_{i\in\Sigma}F_i(V_{n-1}),\qquad V_*=\bigcup_{n=0}^\infty V_n.
\end{equation*}
It is clear that $\{V_n\}_{n\geq0}$ is an increasing sequence of sets and $K$ is the closure of $V_*$ unless $V_0=\emptyset$. It is known that the metric space $(K,d)$ has a fundamental neighborhood system $\{K_{n,x}:\ n\geq0,x\in K\}$, where each $K_{n,x}=\bigcup\limits_{\omega\in \Sigma^n: x\in F_\omega(K)}F_{\omega}(K)$, see \cite[Proposition 1.3.6]{K}. We always assume $\{F_i\}_{i=1}^N$ is p.c.f. and that $(K,d)$ is connected.

\medskip

Our basic assumption on a p.c.f. self-similar set $K$ is the existence of a {\it regular harmonic structure} $(D,{\bf r})$. Denote $Q=\# V_0$. Let ${\bf r}=(r_1,\ldots, r_N)\in (0,\infty)^N$ and $D=(D_{pq})_{p,q\in V_0}$ be a $Q\times Q$ real symmetric matrix such that

(1). for $u\in\ell (V_0)$, $Du=0$ if and only if $u$ is a constant function;

(2). $D_{pq}\geq0$ for any $p,q\in V_0$ with $p\neq q$.

For $u\in\ell (V_0)$, define $\mathcal E_0[u]=-\sum_{p,q\in V_0}D_{p,q}u(p)u(q)$ and
 \begin{equation*}
 \mathcal E_n[u]=\sum_{\omega\in \Sigma^n}\frac{1}{r_\omega}\mathcal E_0[u\circ F_{\omega}|_{V_0}],\qquad u\in\ell(V_n),
 \end{equation*}
 where $n\geq1$ and $r_\omega=r_{\omega_1}\cdots r_{\omega_n}$ for $\omega=\omega_1\cdots\omega_n$ ($r_{\emptyset}=1$).

We say that $(D,{\bf r})$ is a harmonic structure on $(K,\{F_i\}_{i=1}^N,V_0)$ if the following compatible condition holds:
 \begin{equation*}
 \mathcal E_0[u]=\inf_{v\in \ell(V_1),v|_{V_0}=u}\mathcal E_1[v],\qquad \text{ for any } u\in\ell(V_0).
 \end{equation*}
 Moreover, if ${\bf r}\in (0,1)^N$, we call the harmonic structure regular.
Under this condition, $\mathcal E_n[u]$ is an increasing sequence of $n$ and hence for $u\in C(K)$, the space of all continuous functions on $K$, we can define
 \begin{equation*}
 \mathcal E[u]=\lim_{n\rightarrow\infty}\mathcal E_n[u|_{V_n}].
 \end{equation*}
Let $\mathcal F=\{u\in C(K):\ \mathcal E[u]<\infty \}$.

This defines a strongly recurrent self-similar resistance form $(\mathcal E,\mathcal F)$:
\begin{equation}\label{ess}
\mathcal E[u]=\sum_{i=1}^N\frac{1}{r_i}\mathcal E[u\circ F_i],\qquad u\in\mathcal F,
\end{equation}
where $0<r_i<1,i=1,\ldots,N$ are called \textit{energy renormalizing factors}.  By iterating \eqref{ess}, we see that for any $n\geq1$,
\begin{equation}\label{essn}
\mathcal E[u]=\sum_{|\omega|=n}\frac{1}{r_\omega}\mathcal E[u\circ F_\omega],\qquad u\in\mathcal F,
\end{equation}

We call $\mathcal E_{F_\omega(K)}[u]:=\frac{1}{r_\omega}\mathcal E[u\circ F_\omega]$ the {\it energy} of $u$ on the cell  $F_\omega(K)$.

\medskip

We say a function $h\in \mathcal F$ is {\it harmonic} in $K\setminus V_0$ if $$\mathcal E[h]=\inf\{\mathcal E[u]:\ u\in\mathcal F, u|_{V_0}=h|_{V_0}\}.$$
Let $A,B$ be two disjoint nonempty closed subsets of $K$. The {\it effective resistance} $R(A,B)$  between $A$ and $B$ is defined as
\begin{equation*}
R(A,B)^{-1}:=\inf\{\mathcal E[u]:\ u\in\mathcal F, u|_{A}=0,u|_{B}=1\}.
\end{equation*}
The infimum is attained by a unique function, which is harmonic in $K\setminus(A\cup B)$.
We write $R(x,B):=R(\{x\},B)$ and $R(x,y):=R(\{x\},\{y\})$ when $x,y$ are single points.
When we only consider points, by setting $R(x,x)=0$ for all $x\in K$, the resistance $R(\cdot,\cdot)$ is known to be a metric on $K$, which is called the {\it effective resistance metric}. It is known that $R$ is compatible with the topology of $(K,d)$. In addition, $\text{diam}_R(F_{\omega}(K))\asymp r_{\omega}$ for any finite word $\omega$, where $\text{diam}_R(F_{\omega}(K))$ is the diameter of $F_{\omega}(K)$ under $R$.

For a Radon measure $\nu$ supported on $K$, the resistance form $(\mathcal E,\mathcal F)$ turns out to be a Dirichlet form on $L^2(K,\nu)$, which determines a Laplacian $\Delta_\nu$. See  \cite{K,Kig12,S} for details.

\medskip

The problem this paper concerns is generally the following. Assume $(\mathcal E,\mathcal F)$ is a strongly recurrent self-similar resistance form on a p.c.f. self-similar set $K$, and $\Omega$ is a nonempty open connected subset of $K$ with a nonempty boundary $D$. We consider the Dirichlet problem: to find solutions to
\begin{equation*}
\begin{cases}
\Delta u=0 &\text{ in }\Omega,\\
u|_D=f, & f\in C(D).
\end{cases}
\end{equation*}
It is known that the above problem has a unique solution, see for example \cite[Proposition 1.1]{CQ}.

\medskip

We then recall some basic facts in electric network theory. Let $G$ be a finite set, and let $g:
G\times G\rightarrow \mathbb R$ be a nonnegative function such that
\begin{equation*}
g(p,q)=g(q,p),\ g(p,p)=0,\quad  p,q\in G.
\end{equation*}
For $p,q\in G$, we write $p\sim q$ if $g(p,q)>0$, and say that $(G,g)$ is {\it connected} if for any $p,q\in G$ there is a path $p=p_0\sim p_1\sim\cdots\sim p_n=q$. We always assume that $(G,g)$ is connected, and call $(G,g)$ an \textit{electric network}.

For $u\in\ell(G)$, we define the {\it energy} of $u$ on $(G,g)$ to be
\begin{equation*}
\mathcal E_G[u]:=\frac{1}{2}\sum_{p,q\in G}g(p,q)(u(p)-u(q))^2.
\end{equation*}
By polarization, we can define $\mathcal E_G(u,v)=\frac14\left(\mathcal E_G[u+v]-\mathcal E_G[u-v]\right)$ for $u,v\in\ell(G)$.
Then $(\mathcal E_G,\ell(G))$ is a resistance form on $G$ \cite{K}.

For $u\in\ell(G)$, we define the {\it Neumann derivative} of $u$ ({\it flux} of $\nabla u$, the flow associated with $u$; see \cite{B1}) at some vertex $p\in G$ as
\begin{equation}\label{normald}
(du)_p=\sum_{q\in G}g(p,q)(u(p)-u(q)).
\end{equation}
Then clearly, for $u,v\in\ell(G)$,
\begin{equation}\label{DisGuass-Green}
\sum_{p\in G}v(p)(du)_p=\sum_{p\in G}u(p)(dv)_p,
\end{equation}
and in particular,
\begin{equation}\label{zerodiv}
\sum_{p\in G}(du)_p=0.
\end{equation}

For a resistance form $(\mathcal E,\mathcal F)$ on $K$, it is known that the {\it trace} of $\mathcal E$ to a nonempty finite set $V\subset K$ is an electric network $(V,g)$ determined by
\begin{equation*}
\frac12\sum_{p,q\in V}g(p,q)(u(p)-u(q))^2=\min\{\mathcal E[v]:\ v\in\mathcal F, v|_{V}=u\}, \quad u\in\ell(V),
\end{equation*}
while the unique function $v$ minimizing the right hand side is harmonic in $K\setminus V$. In the following, with a little abuse of notation, sometimes we write $(dv)_p$ instead of $(du)_p$ for $p\in V$.




\section{\bf Boundary graph-directed condition}\label{sec3}
\setcounter{equation}{0}\setcounter{theorem}{0}

In this section, for a p.c.f. self-similar set $K$, we will introduce a condition for domains in $K$, named as boundary graph-directed condition, that this paper will be concerned with throughout.

Recall that
graph-directed self-similar sets are generalizations of self-similar sets. Let $(\mathcal A,\Gamma)$ be a {\it directed graph} (allowing loops) with a finite set of {\it vertices} $\mathcal A=\{1,\ldots,P\}$ and a finite set of {\it directed edges} $\Gamma$. For any $\gamma\in\Gamma$, if $\gamma$ is a directed edge from $i$ to $j$ for some $i,j\in\mathcal A$, we set $I(\gamma)=i$ and $T(\gamma)=j$ and call them the {\it initial vertex} and {\it terminal vertex} of $\gamma$ respectively. For $i,j\in\mathcal A$, denote $\Gamma(i)=\{\gamma\in\Gamma:\ I(\gamma)=i\}$ and $\Gamma({i,j})=\{\gamma\in\Gamma: I(\gamma)=i,\ T(\gamma)=j\}$. We assume each $\Gamma(i)$ is nonempty and each  edge $\gamma$ is associated with a contraction $\Phi_{\gamma}$ on $(X,d)$. Then there exists a unique vector of nonempty compact sets $\{D_i\}_{i=1}^{P}$ in $(X,d)$, called \textit{graph-directed self-similar sets} \cite{MW}, satisfying
\begin{equation}\label{eqGD}
D_i=\bigcup_{j=1}^P\bigcup_{\gamma\in\Gamma(i,j)}\Phi_{\gamma}(D_j),\qquad  1\leq i\leq P.
\end{equation}

For $m\geq1$, a finite word ${\bf\gamma}=\gamma_1\gamma_2\cdots\gamma_m$ with $\gamma_i\in\Gamma$ for $ i=1,\ldots,m$ is called {\it admissible} if $T(\gamma_i)=I(\gamma_{i+1})$ for any $i=1,\ldots,m-1$; we set $|{\bf\gamma}|=m$, write $I({\bf\gamma})=I(\gamma_1)$, $T({\bf\gamma})=T(\gamma_m)$, and define $\Phi_{\bf\gamma}=\Phi_{\gamma_1}\circ\cdots\circ\Phi_{\gamma_m}$. We denote by $\Gamma_m$ the set of all admissible words with length $m$, and write $\Gamma_0=\{\emptyset\}$ containing only the empty word by convention. For $0\leq n\leq m$, we denote $[{\bf\gamma}]_n=\gamma_1\cdots\gamma_n$ the $n$-step truncation of ${\bf\gamma}$.  For $i\in\mathcal A$, we also denote $\Gamma_m(i)=\{{\bf\gamma}\in \Gamma_m:\ I({\bf\gamma})=i\}$ ($\Gamma_0(i)=\{\emptyset\}$).
Denote by $\Gamma_*=\bigcup_{m=0}^{\infty}\Gamma_m$ the set of all finite admissible words.

\medskip

We then apply the above definition to a particular situation, domains in p.c.f. self-similar sets.
Let $(K,\{F_i\}_{i=1}^N)$ be a p.c.f. self-similar set. For $P\geq1$, let $\{\Omega_1,\Omega_2,\ldots,\Omega_P\}$ be a vector of connected nonempty open subsets of $K$ such that each $\Omega_i$ has a nonempty boundary with respect to the metric $d$, which is denoted as $D_i$. Later we call $D_i$ the {\it geometric boundary} of $\Omega_i$. We assume that $\Omega_i\neq\Omega_j$ for any $1\leq i<j\leq P$ and $\{(\Omega_i,D_i)\}_{1\leq i\leq P}$ satisfy the following BGD condition (see also \eqref{graphdirected}):

\medskip

\noindent(BGD):\quad \textit{for $1\leq i\leq P$ and $1\leq k\leq N$, if $\Omega_i\cap F_k(K)\neq\emptyset$ and $D_i\cap F_k(K)\neq\emptyset$, then there exists $1\leq j\leq P$ such that
\begin{equation*}
\Omega_i\cap F_k(K)=F_k(\Omega_j),\qquad D_i\cap F_k(K)=F_k(D_j).
\end{equation*}}

According to the configuration of $K$ and $\{\Omega_{i}\}_{i=1}^{P}$, we define the directed graph on $\mathcal{
A}=\{1,\ldots,P\}$ as follows. For each pair $(i,j)$ in the BGD condition, we set $\gamma$ to be a directed edge from $i$ to $j$ with the contraction map $\Phi_\gamma=F_k$. Denote by $\Gamma$ the set of all directed edges $\gamma$ between vertices in $\mathcal A$.
 In this way, we have a directed graph $(\mathcal A,\Gamma)$ and a set of contractions $\{\Phi_\gamma\}_{\gamma\in\Gamma}$ such that for each $\gamma$, there is some $k \in\{1,\ldots,N\}$ satisfying $\Phi_{\gamma}=F_k$.
Furthermore, noting that for any $1\leq i\leq P$ and any $x\in D_i$, there exists $k\in\{1,\ldots,N\}$ such that $x\in F_k(K)$ and $\Omega_i\cap F_k(K)\neq\emptyset$, we see that
$\{D_i\}_{i=1}^{P}$ satisfy the equations \eqref{eqGD} with these $\{\Phi_{\gamma}\}_{\gamma\in\Gamma}$, and hence, $\{D_i\}_{i=1}^{P}$ is a vector of graph-directed self-similar sets.

\medskip

\begin{proposition}\label{BGDprop}
Assume $\{\Omega_i\}_{i=1}^P$ satisfy the BGD condition.

(i). If $\Omega_i\cap V_0\neq\emptyset$, then $\Omega_j\cap V_0\neq\emptyset$ provided that $\Gamma({i,j})\neq\emptyset$;

(ii). There exists $n_0\geq1$ such that $\Omega_{T({\bf\gamma})}\cap V_0\neq\emptyset$ for all $n\geq n_0$ and ${\bf\gamma}\in \Gamma_n$.
\end{proposition}

\begin{proof}
(i). Assume $\Omega_i\cap V_0\neq\emptyset$ and $\gamma\in \Gamma({i,j})$. We consider two possible cases to achieve $\Omega_j\cap V_0\neq \emptyset$.

{\it Case 1. $\Phi_\gamma(\Omega_j)=\Omega_i$.} By $\Omega_i\cap V_0\neq\emptyset$, we can find some $p_k\in\Phi_\gamma(\Omega_j)\cap V_0\subset \Phi_\gamma(V_0)$, where the latter inclusion is implied by $V_0\subset V_1$ and \cite[Proposition 1.3.5(2)]{K}. This implies that $\Phi_\gamma(\Omega_j\cap V_0)\neq\emptyset$ and hence $\Omega_j \cap V_0\neq\emptyset$.

{\it Case 2. $\Phi_\gamma(\Omega_j)\subsetneq\Omega_i$.}  If $\Omega_j\cap V_0=\emptyset$, we must have $\Omega_i\cap\Phi_\gamma(V_0)=\Omega_i\cap  \Phi_\gamma(K)\cap\Phi_\gamma(V_0)=\Phi_\gamma(\Omega_j)\cap \Phi_\gamma(V_0)=\emptyset$, where we used the BGD condition in the second equality. Then since $\Phi_{\gamma}(K\setminus V_0)$ is open in $K$ by \cite[Proposition 1.3.5(2)]{K}, $\Phi_\gamma(\Omega_j)$ and $\Omega_i\setminus\Phi_\gamma(\Omega_j)$ are two nonempty open subsets of $\Omega_i$, which contradicts the connectedness of $\Omega_i$.

(ii). We pick $n\geq1$ sufficiently large such that $\Phi_{\bf\gamma}(\Omega_{T({\bf\gamma})})\subsetneq\Omega_i$ for all $1\leq i\leq P$ and ${\bf\gamma}\in\Gamma_n(i)$. Then the proof is similar to that of Case 2 in (i).
\end{proof}

\begin{proposition}\label{connectedProp}
Assume $\{\Omega_i\}_{i=1}^P$ satisfy the BGD condition. Then each $\Omega_i$ is arcwise connected.
\end{proposition}
\begin{proof}
By \cite[Theorem 1.6.2]{K}, the connectedness of $K$ implies that $K$ and any cell $F_{\omega}(K)$ are arcwise connected. Hence each open set $\Omega_i$ is locally arcwise connected, and so each arcwise connected component of $\Omega_i$ is open.
Since $\Omega_i$ is connected, $\Omega_i$ has only one arcwise connected component,
so that $\Omega_i$ is arcwise connected.
\end{proof}

We will also use the notation of {\it infinite admissible words} ${\bf\gamma}=\gamma_1\gamma_2\cdots$ with $T(\gamma_i)=I(\gamma_{i+1})$ for all $i\geq1$. We denote by $\Gamma_{\infty}$ the collection of all infinite admissible words and $\Gamma_{\infty}(i)=\{{\bf\gamma}=\gamma_1\gamma_2\cdots\in \Gamma_{\infty}:\ I(\gamma_1)=i\}$ for $i=1,\ldots,P$.

For ${\bf\gamma}=\gamma_1\gamma_2\cdots,{\bf\eta}=\eta_1\eta_2\cdots\in\Gamma_{\infty}$  with ${\bf\gamma}\neq{\bf\eta}$, let ${\bf\gamma}\wedge{\bf\eta}$ be the common prefix of ${\bf\gamma}$ and ${\bf\eta}$, then
\begin{equation*}
|{\bf\gamma}\wedge{\bf\eta}|=\min\ \{i\geq1 :\ \gamma_i\neq\eta_i\}-1.
\end{equation*}
Define
\begin{equation*}
\rho({\bf\gamma},{\bf\eta})=\begin{cases}2^{-|{\bf\gamma}\wedge{\bf\eta}|}, &{\bf\gamma}\neq{\bf\eta},\\
                                                 0,& {\bf\gamma}={\bf\eta}.\end{cases}
\end{equation*}
Then by a routine argument, $\rho$ is a metric on $\Gamma_\infty$ and $(\Gamma_\infty,\rho)$ is a compact metric space.

For $i\in\{1,\ldots,P\}$, there is a natural surjective map $$\iota_i: \Gamma_{\infty}(i)\rightarrow D_i$$ given by
$\iota_i({\bf\gamma})=x$ with $\{x\}=\bigcap_{n\geq1}\Phi_{[{\bf\gamma}]_n}(K)$, where $[{\bf\gamma}]_n=\gamma_1\cdots\gamma_n$ is the $n$-step truncation of ${\bf\gamma}$. It is known that $\iota_i$ is continuous (see for example \cite[Theorem 1.2.3]{K}).

\section{\bf Resistance boundary and geometric boundary}\label{sec4}
\setcounter{equation}{0}\setcounter{theorem}{0}

In this section, we will discuss the relation of two ``boundaries" of a domain $\Omega\subset K$,  the geometric boundary and the resistance boundary. We will call them the {\it $d$-boundary} and {\it $R$-boundary} for short.

Let $\Omega$ be a domain in $K$. For a function $u\in C(\Omega)$, by considering $\Omega$ as a countable disjoint union of cells, we define {\it the energy of $u$ on $\Omega$} to be the summation of energies of $u$ on each of the cells and denote it as $\mathcal E_{\Omega}[u]$ (might equal to $+\infty$). By virtue of \eqref{essn}, we see that $\mathcal E_{\Omega}[u]$ does not depend on the partition of disjoint cells in $\Omega$. Denote $\mathcal F_\Omega=\{u\in C(\Omega):\ \mathcal E_\Omega[u]<\infty\}$. By polarization, we define $\mathcal E_{\Omega}(u,v)=\frac{1}{4}\left(\mathcal E_\Omega[u+v]-\mathcal E_\Omega[u-v]\right)$ for $u,v\in\mathcal F_{\Omega}$. It is direct to check that $(\mathcal E_{\Omega},\mathcal F_\Omega)$ is a resistance form on $\Omega$.

Define the {\it effective resistance metric} $R_{\Omega}(x,y)$ for two points $x,y$ in $\Omega$ with respect to $\mathcal E_{\Omega}$: for $x,y\in\Omega$ and $x\neq y$,
\begin{equation*}
R_{\Omega}(x,y)^{-1}:=\inf\{\mathcal E_{\Omega}[u]:\ u\in \mathcal F_\Omega, u(x)=0, u(y)=1\},
\end{equation*}
and $R(x,x)=0$ by convention.
Then $R_{\Omega}(\cdot,\cdot)$ is a metric on $\Omega$ \cite{K}. Let $\widetilde\Omega$ be the completion of $\Omega$ under $R_{\Omega}$, and denote $\partial \Omega=\widetilde\Omega\setminus \Omega$, the {\it $R$-boundary} of $\Omega$. Recall that there is another resistance metric $R(\cdot,\cdot)$ on $\Omega$ inherited from that on $K$.

\begin{lemma}\label{lemma3.0}
Let $A\subset \Omega$ be a nonempty compact subset of $(\Omega,d)$. Then there exists $C>1$ depending on $A$ such that
\begin{equation}\label{RxyRoxy}
R(x,y)\leq R_{\Omega}(x,y)\leq C R(x,y),\quad \forall x,y\in A.
\end{equation}
In particular, the identity map of $\Omega$ is a homeomorphism from $(\Omega,R_{\Omega})$ to $(\Omega,R)$ and $(\Omega,d)$.
\end{lemma}
\begin{proof}
By definition,
\begin{align*}
R_\Omega(x,y)^{-1}&=\inf\{\mathcal E_{\Omega}[u]:\ u\in \mathcal F_\Omega, u(x)=0, u(y)=1\}\\
&\leq \inf\{\mathcal E[u]:\ u\in \mathcal F, u(x)=0, u(y)=1\}=R(x,y)^{-1},
\end{align*}
so we see that $R(x,y)\leq R_{\Omega}(x,y)$.

On the other hand, since $A$ is a compact subset of $\Omega$, fix an $n\geq1$ sufficiently large and a finite number of $n$-cells $\{F_{\omega^{(k)}}(K)\}_{k=1}^m$ such that
\begin{equation*}
A\subset \bigcup_{k=1}^m F_{\omega^{(k)}}(K)\subset \Omega.
\end{equation*}
We can also require that $\bigcup_{k=1}^mF_{\omega^{(k)}}(K)$ is connected by the (arcwise) connectedness of $\Omega$.

For any two points $x,y\in A$, we choose two $n$-cells (may be equal), say $F_{\omega}(K)$ and $F_{\omega'}(K)$, in $\{F_{\omega^{(k)}}(K)\}_{k=1}^m$ such that $x\in F_{\omega}(K)$ and $y\in F_{\omega'}(K)$. Let $u$ be the unique function in $\mathcal F_\Omega$ such that $\mathcal E_{\Omega}[u]=R_{\Omega}(x,y)^{-1}$ and $u(x)=0$, $u(y)=1$.  Define a function $v\in\mathcal F$ such that $v|_{F_{\omega^{(k)}}(K)}=u|_{F_{\omega^{(k)}}(K)}$ for each $1\leq k\leq m$, $v=0$ on $V_n\setminus \left({\bigcup_{k=1}^mF_{\omega^{(k)}}(V_0)}\right)$ and $v\circ F_\tau$ is harmonic in $K\setminus V_0$ for each $\tau\in \Sigma^n\setminus\{\omega^{(1)},\ldots,\omega^{(m)}\}$.
Then $v(x)=0$, $v(y)=1$ and
\begin{equation}\label{eqRxy}
R(x,y)^{-1}\leq \mathcal E[v]
=\sum_{k=1}^m\mathcal E_{F_{\omega^{(k)}}(K)}[u]+\sum_{\tau\in\Sigma^n\setminus\{\omega^{(1)},\ldots,\omega^{(m)}\}}\mathcal E_{F_{\tau}(K)}[v].
\end{equation}
Since $v$ attains values $0$ and $1$ in the cells $F_{\omega}(K)$ and $F_{\omega'}(K)$ respectively, and the union of cells $\bigcup_{k=1}^mF_{\omega^{(k)}}(K)$ is connected, we see that $\sum_{k=1}^m\mathcal E_{F_{\omega^{(k)}}(K)}[u]\geq C_1$ for some $C_1>0$ depending on $n$ and $\{r_i\}_{i=1}^N$. Also noting that $0\leq v\leq 1$ in each of the cells $F_{\tau}(K)$ for $\tau\in\Sigma^n\setminus\{\omega^{(1)},\ldots,\omega^{(m)}\}$, and $v$ is harmonic in each $F_{\tau}(K\setminus V_0)$, we obtain $\sum_{\tau\in\Sigma^n\setminus\{\omega^{(1)},\ldots,\omega^{(m)}\}}\mathcal E_{F_{\tau}(K)}[v]\leq C_2$ for some $C_2>0$ depending on $n$ and $\{r_i\}_{i=1}^N$. Hence the right-hand side of \eqref{eqRxy} is bounded from above by
\begin{equation}\label{eqROxy}
C\sum_{k=1}^m\mathcal E_{F_{\omega^{(k)}}(K)}[u]\leq C\mathcal E_{\Omega}[u]= CR_{\Omega}(x,y)^{-1},
\end{equation}
for some $C>1$ depending on $n$ and $\{r_i\}_{i=1}^N$.

Combining \eqref{eqRxy} and \eqref{eqROxy}, we obtain the second inequality of \eqref{RxyRoxy}.

From  \eqref{RxyRoxy}, we see that the identity map of $\Omega$ is a homeomorphism from $(\Omega,R_{\Omega})$ to $(\Omega,R)$ and so also to $(\Omega,d)$.
\end{proof}

Let $\{\Omega_i\}_{i=1}^P$ be a finite collection of domains in $K$ with $d$-boundaries $\{D_i\}_{i=1}^{P}$ satisfying the BGD condition.

We say a (finite or infinite) sequence of cells $\{F_{\omega^{(k)}}(K)\}_{k\geq1}$  {\it a chain of cells} if $F_{\omega^{(k)}}(K)\cap F_{\omega^{(k+1)}}(K)\neq\emptyset$ for all $k\geq1$.
For a finite chain of cells $\{F_{\omega^{(k)}}(K)\}_{k=1}^m$ with $x\in F_{\omega^{(1)}}(K)$ and $y\in F_{\omega^{(m)}}(K)$, we say it {\it connects $x$ and $y$}.

\begin{lemma}\label{connectedlemma}
There exists $n_1\geq1$ such that
for each $\Omega_i$ with $\Omega_i\cap V_1\neq\emptyset$ and $x,y\in \Omega_i\cap V_1$, there exists a chain of $n_1$-cells $\{F_{\omega^{(k)}}(K)\}_{k=1}^m$ in $\Omega_i$ connecting $x$ and $y$.
\end{lemma}

\begin{proof}
By Proposition \ref{connectedProp}, each $\Omega_i$ is arcwise connected. Hence for any $x,y\in \Omega_i\cap V_1$, there exists a curve joining $x$ and $y$ in $\Omega_i$ (a continuous map  $f:\ [0,1]\rightarrow \Omega_i$ such that $f(0)=x$, $f(1)=y$). By that $\Omega_i$ is open in $K$, an $\varepsilon$-neighborhood of the curve is contained in $\Omega_i$, which gives a desired chain of $n$-cells for large $n$. Since the numbers of $\Omega_i$ and pairs $x,y\in\Omega_i\cap V_1$ are finite, we see that there exists a common $n_1$ as required.
\end{proof}

In the following, we write $r_{\max}=\max\{r_i:\ 1\leq i\leq N\}$ and $r_{\min}=\min\{r_i:\ 1\leq i\leq N\}$. 

\begin{proposition}\label{lemma4.2}
Each $(\Omega_i,R_{\Omega_i})$ is a bounded metric space.
\end{proposition}
\begin{proof}
By Proposition  \ref{BGDprop}(ii), we choose $n_0$ such that $\Omega_{T({\bf\gamma})}\cap V_0\neq\emptyset$ for all ${\bf\gamma}\in \Gamma_{n}$  with $n\geq n_0$. Let
\begin{equation*}
\mathcal B=\{T({\bf\gamma}):\ {\bf\gamma}\in\Gamma_{n},n\geq n_0\}.
\end{equation*}

 We first prove that $(\Omega_i,R_{\Omega_i})$ is bounded for each $i\in \mathcal B$.

For any $x=\pi(\omega)\in\Omega_i$ with $\omega\in\Sigma^\infty$, let $m\geq0$ be such that $F_{[\omega]_{m+1}}(K)\subset \Omega_i$ and $F_{[\omega]_m}(K)\nsubset \Omega_i$.
By the BGD condition, $F_{[\omega]_m}(K)\cap \Omega_i=\Phi_{\bf\gamma}(\Omega_{T({\bf\gamma})})$ for some ${\bf\gamma}\in\Gamma_{m}(i)$ ($\Phi_{\emptyset}=\rm{Id}$, $T(\emptyset)=i$). So $x\in F_{[\omega]_{m+1}}(K)\subset \Phi_{\bf\gamma}(\Omega_{T({\bf\gamma})})$.

Then by Lemma \ref{connectedlemma}, we have the following two facts.

\medskip

{\it Fact 1. For $0\leq k\leq m$, and $y\in\Phi_{[{\bf\gamma}]_{k}}\left(\Omega_{T([{\bf\gamma}]_{k})}\cap V_0\right)$, $z\in\Phi_{[{\bf\gamma}]_{k+1}}\left(\Omega_{T([{\bf\gamma}]_{k+1})}\cap V_0\right)$ ($F_{[\omega]_{m+1}}(V_0)$ if $k=m$), there exists a chain of $(n_1+k)$-cells in $\Phi_{[{\bf\gamma}]_{k}}\left(\Omega_{T([{\bf\gamma}]_{k})}\right)$ connecting $y$ and $z$.}

\medskip

{\it Fact 2.
For $k\geq m+1$, and $y\in F_{[\omega]_k}(V_0)$, $z\in F_{[\omega]_{k+1}}(V_0)$, there exists a chain of $(n_1+k)$-cells in $F_{[\omega]_k}(K)$ connecting $y$ and $z$.}

\medskip

Note that the number of cells in each above chain is bounded from above by $M=N^{n_1}$. For convenience, we adjust the number of cells in each above chain to be $M$ by adding some repeated cells in the chain.
Hence from these two facts, there exists a chain of cells $\{F_{\omega^{(k)}}(K)\}_{k=1}^\infty$ in $\Omega_i$ and a sequence of points $\{x_k\}_{k=0}^\infty$ with $x_0\in F_{\omega^{(1)}}(V_0)\cap V_0$, $x_k\in F_{\omega^{(k)}}(V_0)\cap F_{\omega^{(k+1)}}(V_0)$ for $k\geq1$ and $\lim_{k\rightarrow\infty}x_k=x$ (w.r.t. $d$), and $|\omega^{((l-1)M+1)}|=|\omega^{((l-1)M+2)}|=\cdots=|\omega^{(l M)}|=n_1+l-1$ for each $l\geq1$. We have for any $u\in\mathcal F_{\Omega_{i}}$,
\begin{align*}
|u(x)-u(x_{0})|&\leq \sum_{k=0}^\infty|u(x_k)-u(x_{k+1})|\\
&\lesssim\sum_{k=1}^\infty r^{1/2}_{\omega^{(k)}}\mathcal E_{F_{\omega^{(k)}}(K)}[u]^{1/2}
\leq\sum_{k=1}^\infty r^{{(n_1+k/M-1)}/{2}}_{\max}\mathcal E_{\Omega_i}[u]^{1/2}
\lesssim\mathcal E_{\Omega_i}[u]^{1/2}.
\end{align*}
Hence we see that for $x,y\in\Omega_i$,
\begin{equation*}
|u(x)-u(y)|\leq |u(x)-u(x_0)|+|u(y)-u(x_0)|\leq C \mathcal E_{\Omega_i}[u]^{1/2}
\end{equation*}
for some constant $C>0$ independent of $u$, $i$, $x$ and $y$. This gives that
$\Omega_i$ is bounded under $R_{\Omega_i}$ for $i\in \mathcal B$.

Finally, for $i\in \mathcal A$, noticing that for each ${\bf\gamma}\in\Gamma_{n_0}(i)$, $T({\bf\gamma})\in\mathcal B$, and $\Omega_{T({\bf\gamma})}$ is already bounded, by a similar chain argument as above, $\Omega_i$ is also bounded under $R_{\Omega_i}$.
\end{proof}

\begin{theorem}\label{thm3.1}
For $i=1,\ldots,P$, $(\partial \Omega_i,R_{\Omega_i})$ is homeomorphic to $(\Gamma_{\infty}(i),\rho)$, and $(\widetilde\Omega_i,R_{\Omega_i})$ is compact.
\end{theorem}

\begin{proof}
For ${\bf \gamma}\in \Gamma_*$, we write $\Omega_{\bf\gamma}:=\Phi_{\bf\gamma}(\Omega_{T({\bf\gamma})})$ for brevity. Let $x\in \partial\Omega_i$, and  $\{x_n\}_{n\geq1}$ be a sequence in $\Omega_i$ such that $\lim_{n\rightarrow\infty}x_n=x$ w.r.t. $R_{\Omega_i}$.

\medskip

{\it Claim $(*)$. For any $m\geq1$, there exists a unique ${\bf\gamma}\in\Gamma_m(i)$ such that $x_n\in \Omega_{\bf\gamma}$ for all large enough $n$.}

\medskip

For $m\geq1$, denote \begin{equation*}
U_{i,m}=\bigcup_{{\bf\gamma}\in \Gamma_m(i)}\Omega_{\bf\gamma}.
\end{equation*}
We prove this claim through two steps.

First we prove that, for any $m$, it always holds that $x_n\in U_{i,m}$ for all large enough $n$. Otherwise, there exists $m_0\geq1$ and a subsequence $\{x_{n_k}\}$ contained in $\Omega_i\setminus U_{i,m_0}$ which converges to $x$ w.r.t. $R_{\Omega_i}$. This gives that $x$ is in the closure of $\Omega_i\setminus U_{i,m_0}$ under $R_{\Omega_i}$, which is contained in $\Omega_i$ by using Lemma \ref{lemma3.0}, a contradiction to $x\in\partial\Omega_i$.

Next, we turn to proving Claim $(*)$. If it does not hold, then  we can pick $m_1\geq1$ and ${\bf\gamma}\neq{\bf\eta}\in\Gamma_{m_1}(i)$ such that both $\Omega_{\bf\gamma}$ and $\Omega_{\bf\eta}$ contain infinitely many elements in the sequence $\{x_n\}$. We may require that
$\Omega_{\bf\gamma}\cap\Omega_{\bf\eta}=\emptyset$ in addition, since otherwise, we can replace $\bf\gamma$ and $\bf\eta$ by their offsprings $\bf\gamma'$ and $\bf\eta'$ such that $\Omega_{\bf\gamma'}\cap\Omega_{\bf\eta'}=\emptyset$ and replace $m_1$ by a larger number $m_2=|{\bf\gamma'}|=|{\bf\eta'}|$.
Considering a function $u$ on $\Omega_i$ such that $u|_{\Omega_{\bf\gamma}}=0$, $u|_{\Omega_{\bf\eta}}=1$ and harmonic elsewhere in each $m_1$-cell in $\Omega_i$,
we have
\begin{equation}\label{Rolowbd}
R_{\Omega_i}(y,z)\geq \mathcal E_{\Omega_i}[u]^{-1}>0
\end{equation}
 for all $y\in\Omega_{\bf\gamma}$ and $z\in\Omega_{\bf\eta}$. This contradicts that $\{x_n\}$ is a Cauchy sequence. So Claim $(*)$ holds.

\medskip

By Claim $(*)$,  each sequence $\{x_n\}$ converging to $x\in \partial\Omega_i$ determines a unique infinite admissible word ${\bf\gamma}\in\Gamma_\infty(i)$. For two sequences $\{x_n\}$, $\{y_n\}$, if they determine two distinct words ${\bf\gamma},{\bf\eta} \in\Gamma_\infty(i)$, then they must converge to distinct points in $\partial\Omega_i$, since from the above paragraph, $R_{\Omega_i}(x_n,y_n)\geq c_0>0$ for some $c_0>0$ and all large enough $n$. So for each $x\in\partial\Omega_i$, it determines a unique word ${\bf\gamma}\in \Gamma_\infty(i)$, we denote it as $\mathcal T(x)$.

For ${\bf\gamma}\in \Gamma_{\infty}(i)$, we pick a sequence $\{x_n\}$ in $\Omega_i$ such that $x_n\in \Phi_{[{\bf\gamma}]_n}({K})$ for all $n$. Note that
\begin{equation*}
R_{\Omega_i}(x_n,x_{n+1})\leq r_{[{\bf\gamma}]_n}R_{\Omega_{T([{\bf\gamma}]_n)}}\left(\Phi_{[{\bf\gamma}]_n}^{-1}(x_n),\Phi_{[{\bf\gamma}]_n}^{-1}(x_{n+1})\right)\leq C_1r_{\max}^n
\end{equation*}
for some $C_1>0$ by Proposition \ref{lemma4.2}, where $r_{[{\bf\gamma}]_n}=r_{\omega}$ with $\omega\in\Sigma^n$ uniquely determined by $F_{\omega}=\Phi_{[{\bf\gamma}]_n}$ under the BGD condition.
We see that $\{x_n\}$ is a Cauchy sequence w.r.t. $R_{\Omega_i}$ and has a limit $x$ in $\widetilde\Omega_i$. However, due to Lemma \ref{lemma3.0}, $x\notin \Omega_i$. Hence $\mathcal T$ is a surjection.

Now we prove that $\mathcal T$ is a homeomorphism between $(\partial\Omega_i,R_{\Omega_i})$ and $(\Gamma_\infty(i),\rho)$.
Pick $x\neq y\in \partial\Omega_i$, denote $\mathcal T(x)={\bf\gamma}$, $\mathcal T(y)={\bf\eta}$ and $\bf\kappa={\bf\gamma}\wedge{\bf\eta}$. Let $\{x_n\},\{y_n\}$ be two sequences converging to $x$ and $y$ respectively. Since $x_n,y_n\in\Omega_{\bf\kappa}$ for all large enough $n$, we have by Proposition \ref{lemma4.2},
\begin{equation}\label{HoUp}
R_{\Omega_i}(x_n,y_n)\leq C_1r_{\max}^{|\bf\kappa|}=C_1\rho({\bf\gamma},{\bf\eta})^{-\log r_{\max}/\log2}.
\end{equation}
On the other hand, since $\{x_n\}$, $\{y_n\}$ will enter two disjoint offsprings of $\Omega_{\bf\kappa}$  for all large $n$, by using the same argument as \eqref{Rolowbd}, we have
\begin{equation}\label{HoDown}
R_{\Omega_i}(x_n,y_n)\geq C_2r_{\min}^{|\bf\kappa|}= C_2\rho({\bf\gamma},{\bf\eta})^{-\log r_{\min}/\log2}
\end{equation}
for some $C_2>0$.
By combining \eqref{HoUp} and \eqref{HoDown} and letting $n\rightarrow\infty$, we have
\begin{equation}\label{eq-homeomorphism}
C_2\rho({\bf\gamma},{\bf\eta})^{-\log r_{\min}/\log2}\leq R_{\Omega_i}(x,y)\leq C_1\rho({\bf\gamma},{\bf\eta})^{-\log r_{\max}/\log2}.
\end{equation}
Combining this with the fact that $\mathcal T$ is a surjection (and an injection by \eqref{eq-homeomorphism}), we have $(\partial\Omega_i,R_{\Omega_i})$ is homeomorphic to $(\Gamma_{\infty}(i),\rho)$ with a homeomorphism $\mathcal T$.

\medskip

Since $(\Gamma_{\infty}(i),\rho)$ is compact, by homeomorphism, $(\partial\Omega_i,R_{\Omega_i})$ is also compact. For a sequence $\{x_n\}$ in $\widetilde\Omega_i$, we prove that it must contain a converging subsequence. We separate the proof into the following two cases.

 {\it (a). $\liminf_{n\rightarrow\infty}\inf\{R_{\Omega_i}(x_n,y):\ y\in\partial\Omega_i\}=0$.}

  Then there is a subsequence $\{x_{n_k}\}$ such that $\lim_{k\rightarrow\infty}\inf\{R_{\Omega_i}(x_{n_k},y):\ y\in\partial\Omega_i\}=0$. By the compactness of $\partial\Omega_i$, there exists a converging subsequence of $\{x_{n_k}\}$.

{\it  (b). $\liminf_{n\rightarrow\infty}\inf\{R_{\Omega_i}(x_n,y):\ y\in\partial\Omega_i\}>0$. }
  
  Then by Claim $(*)$, we see that there exists $m_0\geq1$ such that $\{x_n\}$ has a subsequence $\{x_{n_k}\}$ contained in $\Omega_i\setminus U_{i,m_0}$. Noting that $\Omega_i\setminus U_{i,m_0}$ is a union of finite $m_0$-cells of $K$ and is compact under $R_{\Omega_i}$, thus $\{x_{n_k}\}$ has a converging subsequence.

Above all, any sequence in $\widetilde\Omega_i$ has a converging subsequence, hence $(\widetilde\Omega_i,R_{\Omega_i})$ is compact.
\end{proof}
In a recent work \cite[Theorem 4.5]{KT22}, Kigami and Takahashi obtained a similar result on a particular $\Omega$, the SG minus its bottom line, by utilizing the binary tree structure of that domain.

\medskip

{\it \noindent{\bf Remark 1.} Recall that there is a continuous surjective map $\iota_i$ from $(\Gamma_{\infty}(i),\rho)$ to $(D_i,d)$. This induces a continuous surjective map, still denoted as $\iota_i$, from the $R$-boundary $(\partial \Omega_i,R_{\Omega_i})$ to the $d$-boundary $(D_i,d)$. Hence if $f$ is a continuous function on $D_i$, $f\circ\iota_i$ is a continuous function on $\partial\Omega_i$.

\medskip

\noindent{\bf Remark 2.} For ${\bf\gamma}\in\Gamma_*$, writing $I({\bf\gamma})=i$, $T({\bf\gamma})=j$, we define $\theta_{\bf\gamma}:\ \Gamma_\infty(j)\rightarrow \Gamma_\infty(i)$ by $\theta_{\bf\gamma}({\bf\eta})$ being the concatenation of ${\bf\gamma}$ and ${\bf\eta}$ for each ${\bf\eta}\in\Gamma_\infty(j)$. By Theorem \ref{thm3.1}, with a slight abuse of notation, $\theta_{\bf\gamma}$ can be viewed as a continuous injection
from $\partial\Omega_{j}$ into $\partial\Omega_{i}$. It is direct to check that
\begin{equation*}
\Phi_{\bf\gamma}\circ\iota_{j}=\iota_{i}\circ\theta_{\bf\gamma}.
\end{equation*}
Actually, the map $\widetilde\Phi_{\bf\gamma}:(\widetilde\Omega_j,R_{\Omega_j})\rightarrow(\widetilde\Omega_i,R_{\Omega_i})$ given by $\Phi_{\bf\gamma}$ on $\Omega_j$ and $\theta_{\bf\gamma}$ on $\partial\Omega_j$ is the continuous extension of $\Phi_{\bf\gamma}|_{\Omega_j}$.

Also, $\{\partial\Omega_i\}_{i=1}^P$ satisfy the decomposition
\begin{equation}\label{RboundaryGD}
\partial\Omega_i=\bigcup_{j=1}^P\bigcup_{\gamma\in\Gamma(i,j)}\theta_{\gamma}(\partial\Omega_j),\qquad  1\leq i\leq P.
\end{equation}

}

%

\medskip

By a standard theory of resistance forms \cite[Theorem 2.3.10]{K}, the resistance form $(\mathcal E_{\Omega},\mathcal F_{\Omega})$ on $\Omega$ (we omit the subscripts $i$ for short) extends to be a resistance form on $\widetilde\Omega$, where each function in $\mathcal F_\Omega$ is continuously extended to $\widetilde\Omega$.
Furthermore, for a Radon measure $\nu$ on $\widetilde\Omega$ with full support, $(\mathcal E_{\Omega},\mathcal F_{\Omega})$ will generate a Dirichlet form on $L^2(\widetilde\Omega,\nu)$, which is associated with a nonpositive self-adjoint operator $\Delta_\nu$ called the Laplacian on $\widetilde\Omega$ (see \cite[Theorem 2.4.2]{K} or \cite[Theorem 9.4]{Kig12}).

For a nonempty closed subset $A\subset \widetilde\Omega$, by \cite[Lemma 8.2 and Theorem 8.4]{Kig12}, for any $u_0\in \mathcal F_\Omega|_A=\{v|_A:\ v\in\mathcal F_\Omega\}$, there exists a unique $u\in\mathcal F_\Omega$ such that $u|_A=u_0$ and
\begin{equation*}
\mathcal E_{\Omega}[u]=\min\{\mathcal E_\Omega[v]:\ v\in\mathcal F_{\Omega},v|_A=u_0\}.
\end{equation*}
The function $u$ also satisfies
\begin{equation*}
\Delta_\nu u=0,\quad \text{ in }\widetilde\Omega\setminus A,
\end{equation*}
in the weak sense, i.e.
$\mathcal E_\Omega(u,v)=0$, for any $v\in \mathcal F_\Omega, v|_A=0$.
Say the function $u$ is {\it harmonic} in $\widetilde\Omega\setminus A$, and call $u$ a {\it harmonic extension} of $u_0$ from $A$ to $\widetilde\Omega$. In particular, when $A=\partial\Omega$, we say $u$ is a harmonic function in $\Omega$ with boundary values $u_0$.

In a standard way \cite[Theorem 4.3]{Kig12}, for a finite collection  $G=\{A_1,\ldots,A_m\}$ of nonempty disjoint closed sets in $\widetilde\Omega$, we can take a slightly different kind of trace of $(\mathcal E_{\Omega},\mathcal F_{\Omega})$ to get a ``shorted" resistance form $(\mathcal E_G,\ell(G))$ on $G$ (viewing $G$ as a finite point set). Precisely, by identifying $\ell(G)$ with
\begin{equation*}
\{u|_{\bigcup_{i=1}^mA_i}:\ u\in\mathcal F_{\Omega}, u \text{ takes constant values on each } A_i\},
\end{equation*}
define for $f,g\in\ell(G)$, $\mathcal E_G(f,g):=\mathcal E_{\Omega}(H^Gf,H^Gg)$, where $H^Gf$ (or $H^Gg$) is the unique harmonic extension of $f$ (or $g$) from ${\bigcup_{i=1}^mA_i}$ to $\widetilde\Omega$.
Such a trace will induce an electric network on $G$.

With a little abuse of notation, sometimes for $f\in\ell(G)$, we write $(d H^Gf)_p$ instead of $(d f)_p$, where $p$ stands for some $A_i$.
By compatibility, we have $(dH^Gf)_p=(dH^{G'}(H^Gf|_{G'}))_p$ if $G'\supset G$ and $G'\setminus G$ is a finite subset of $\widetilde\Omega\setminus\bigcup_{k=1}^mA_k$.

In later sections, we always take $G$ to be a collection of sets in the form of $\theta_{\bf\gamma}(\partial\Omega_{T({\bf\gamma})})$ with ${\bf\gamma}\in\Gamma_*$ together with some single points in $\Omega$.

\section{\bf Hitting probability}\label{sec5}
\setcounter{equation}{0}\setcounter{theorem}{0}

Let $(K,\{F_i\}_{i=1}^N,V_0)$ be a p.c.f. self-similar set with  $V_0=\{p_1,\ldots,p_Q\}$ for some $Q\geq2$. Let $(\mathcal E, \mathcal F)$ be a strongly recurrent self-similar resistance form on $K$ satisfying \eqref{ess}. 

Let $P\geq1$ and $\{\Omega_i\}_{i=1}^{P}$ be a finite collection of domains in $K$ with $d$-boundaries $\{D_i\}_{i=1}^{P}$ satisfying the BGD condition \eqref{graphdirected}. Denote by $\{\partial\Omega_i\}_{i=1}^{P}$ the corresponding $R$-boundaries of $\{\Omega_i\}_{i=1}^{P}$ respectively.
Denote $\mathcal A=\{1,\ldots,P\}$. 
%
%


\medskip

\noindent{\bf Flux transfer matrices.} Let $(\mathcal A,\Gamma)$ be the directed graph induced from the BGD condition.
 For each $\gamma\in\Gamma(i,j)$, notice that by BGD, there is a contraction map $\Phi_\gamma$ such that $\Phi_\gamma(\Omega_j)\subset \Omega_i$.
 In the following, we associate each $\gamma$ with a $Q\times Q$ real matrix $M_\gamma$, whose $(k,\ell)$-entry represents:

  \medskip

  {\it the flux of the unit flow on $\widetilde\Omega_i$ from $\partial\Omega_i$ to $p_k$ through $\Phi_\gamma(p_\ell)$ outwards from $\Phi_{\gamma}(\Omega_j)$.}

\medskip

For any $1\leq k\leq Q$, if $p_k\notin \Omega_i$, we simply set the $k$-th row of $M_\gamma$ to be zeros; otherwise, if $p_k\in \Omega_i$,
let $\varphi$ be the realization of $R_{\Omega_i}(\partial\Omega_i,p_k)$, i.e. $\varphi$ is the unique function on $\widetilde \Omega_i$ such that $\varphi|_{\partial\Omega_i}=0$, $\varphi(p_k)=1$ and $\varphi$ is harmonic in $\Omega_i\setminus \{p_k\}$ with $\mathcal E_{\Omega_i}[\varphi]=R_{\Omega_i}(\partial\Omega_i,p_k)^{-1}$. Let
\begin{equation}\label{vk}
v:=R_{\Omega_i}(\partial\Omega_i,p_k)\varphi,
\end{equation}
then $v$ satisfies $(dv)_{p_k}=1$. Since $\gamma\in\Gamma({i,j})$, we have $\Phi_{\gamma}(\Omega_j)\subset \Omega_i$. Consider the restriction of the function $v$ on $\Phi_{\gamma}(\Omega_j)$, and denote it as $\tilde v$, then $\tilde v$ is harmonic in $\Phi_{\gamma}(\Omega_j)\setminus \Phi_{\gamma}(V_0)$ in the sense of $\mathcal E_{\Phi_{\gamma}(\Omega_j)}$. Now by tracing the resistance form $\mathcal E_{\Phi_{\gamma}(\Omega_j)}$ to the finite set $\Phi_{\gamma}(V_0)\cup\{\theta_{\gamma}(\partial\Omega_{j})\}$ (i.e. the points in $\Phi_{\gamma}(V_0)$ together with the ``singleton" obtained by shorting $\theta_{\gamma}(\partial\Omega_{j})$), we obtain a electric network, see Remark 2 in Section \ref{sec4}. Based on this electric network, we define
\begin{equation}\label{fluxtransmtx}
M_\gamma(k,\ell)= \begin{cases}(d\tilde v)_{\Phi_{\gamma}(p_\ell)}=\frac{1}{r_\omega}\mathcal E_{\Omega_j}\left(v\circ\Phi_{\gamma},H^{(\Omega_j\cap V_0)\cup\{\partial\Omega_j\}}({\bf1}_{p_\ell})\right)\quad &\text{ if $p_\ell\in \Omega_j$},\\
0 \quad &\text{ if $p_\ell\notin \Omega_j$},
\end{cases}
\end{equation}
where $\omega\in\{1,\ldots,N\}$ is such that $F_\omega=\Phi_\gamma$ and $H^{(\Omega_j\cap V_0)\cup\{\partial\Omega_j\}}({\bf1}_{p_\ell})\in\mathcal F_{\Omega_j}$ denotes the function on $\widetilde\Omega_j$ which is $1$ at $p_\ell$, $0$ on $\partial\Omega_{j}$ and on $\Omega_j\cap(V_0\setminus\{p_\ell\})$ and harmonic in $\Omega_j\setminus V_0$ with respect to $(\mathcal E_{\Omega_j},\mathcal F_{\Omega_j})$.



\medskip

We call $\{M_\gamma\}_{\gamma\in\Gamma(i)}$ the {\it flux transfer matrices} associated with domain $\Omega_i$.

\begin{proposition}
For  $1\leq i\leq P$ and $1\leq k\leq Q$ such that $p_k\in\Omega_i\cap V_0$, we have
\begin{equation*}
\sum_{\ell=1}^QM_{\gamma}(k,\ell)>0, \quad \forall \gamma\in\Gamma(i),
\end{equation*}
 and
\begin{equation}\label{eqeq3.1}
\sum_{\gamma\in\Gamma({i})}\sum_{\ell=1}^QM_{\gamma}(k,\ell)=1.
\end{equation}
\end{proposition}
\begin{proof}
For ${\gamma}\in \Gamma(i)$, by the strong maximum principle (see \cite[Theorem 4.1]{Kuw}), $v>0$ (defined in \eqref{vk})  on $\Phi_\gamma(\Omega_{T(\gamma)}\cap V_0)\subset\Omega_i$. Since $v|_{\partial \Omega_i}=0$, we see that $(d\tilde v)_{\theta_{\gamma}(\partial\Omega_{T(\gamma)})}<0$. Then on the electric network given by tracing the resistance form $\mathcal E_{\Phi_{\gamma}(\Omega_{T(\gamma)})}$ to the finite set $\Phi_{\gamma}(\Omega_{T(\gamma)}\cap V_0)\cup\{\theta_{\gamma}(\partial\Omega_{T(\gamma)})\}$,  by \eqref{zerodiv}, we have
\begin{equation*}
(d\tilde v)_{\theta_{\gamma}(\partial\Omega_{T(\gamma)})}+\sum_{\ell: p_\ell\in \Omega_{T(\gamma)}}(d\tilde v)_{\Phi_{\gamma}(p_\ell)}=0,
\end{equation*}
which gives that $\sum_{\ell=1}^QM_{\gamma}(k,\ell)=-(d\tilde v)_{\theta_{\gamma}(\partial\Omega_{T(\gamma)})}>0$.

\medskip

To prove \eqref{eqeq3.1}, for $k$ with $p_k\in\Omega_i\cap V_0$, we trace the resistance form $\mathcal E_{\Omega_i}$ to the finite set $\{p_k\}\cup\{\theta_{\gamma}(\partial\Omega_{T(\gamma)}):\ \gamma\in\Gamma(i)\}$ to get an  electric network. By using \eqref{zerodiv} again, we obtain
\begin{equation*}
(dv)_{p_k}+\sum_{\ell: p_\ell\in \Omega_{T(\gamma)}}(d v)_{\theta_{\gamma}(\partial\Omega_{T(\gamma)})}=0,
\end{equation*}
which together with $(dv)_{p_k}=1$ and $(d v)_{\theta_{\gamma}(\partial\Omega_{T(\gamma)})}=(d\tilde v)_{\theta_{\gamma}(\partial\Omega_{T(\gamma)})}=-\sum_{\ell=1}^QM_{\gamma}(k,\ell)$ yields \eqref{eqeq3.1}.
\end{proof}

From now, for brevity of notation, for ${\bf\gamma}\in\Gamma_m$, $m\geq1$, we write
\begin{equation}\label{eq5.12}
\Omega_{\bf\gamma}:=\Phi_{\bf\gamma}(\Omega_{T({\bf\gamma})}), \quad D_{\bf\gamma}:=\Phi_{\bf\gamma}(D_{T({\bf\gamma})}),\quad \partial\Omega_{\bf \gamma}:=\theta_{\bf \gamma}(\partial\Omega_{T(\bf \gamma)}).
\end{equation}
Noting that by \eqref{RboundaryGD}, we have
\begin{equation}\label{eqeq1}
\partial\Omega_i=\bigcup_{{\bf \gamma}\in\Gamma_m(i)}\partial\Omega_{\bf \gamma},\qquad \text{for all $m\geq1$},
\end{equation}
where the union is disjoint.

Now for those $\Omega_i$ with $\Omega_i\cap V_0\neq\emptyset$, we will use the matrices $\{M_\gamma\}_{\gamma\in\Gamma(i)}$ to construct a class of positive Borel measures $\{\mu_{i,k}:\ p_k\in \Omega_i\cap V_0,1\leq k\leq Q \}$ on $\partial\Omega_i$.

\begin{definition}\label{defhmeasure}
For ${\bf \gamma}=\gamma_1\cdots\gamma_m\in\Gamma_m(i)$, write $M_{\bf\gamma}=M_{\gamma_1}\cdots M_{\gamma_m}$.
We define
\begin{equation}\label{harmonicm}
\mu_{i,k}(\partial\Omega_{\bf \gamma})={\bf e}^t_kM_{\bf \gamma}{\bf 1},
\end{equation}
where ${\bf e}_k=(0,\ldots,1,\ldots,0)^t$ is the $Q$-dimensional unit vector whose $k$-th coordinate is $1$, and ${\bf 1}$ is the $Q$-dimensional vector with all entries equal to $1$.
\end{definition}
Note that $\mu_{i,k}(\partial\Omega_{\bf \gamma})$ is the summation of the $k$-th row of $M_{\bf\gamma}$.

\medskip

Let $v=R_{\Omega_i}(\partial\Omega_i,p_k)\varphi$ as above. Let $i\in\mathcal A$, $m\geq0$.
The trace of the energy $\mathcal E_{\Omega_i}$ to $\{p_k\}\cup\{\partial\Omega_{\bf \gamma}:\ {\bf \gamma}\in\Gamma_m(i)\}$ is an electric network.  The following lemma gives the relation between the Neumann derivative $(dv)_{\partial\Omega_{\bf\gamma}}$ and $\mu_{i,k}$.

\begin{lemma}\label{lemma5.4}
For any ${\bf \gamma}\in\Gamma_m(i)$, $(dv)_{\partial\Omega_{\bf\gamma}}=-\mu_{i,k}(\partial\Omega_{\bf\gamma})$.
\end{lemma}

\begin{proof}
Let $v|_{\Omega_{\bf\gamma}}$ denote the restriction of $v$ on $\Omega_{\bf\gamma}$. We claim that for any $1\leq\ell\leq Q$ with $p_\ell\in\Omega_{T({\bf\gamma})}$,
\begin{equation}\label{eqfluxdvgamma1}
\left(d(v|_{\Omega_{\bf\gamma}})\right)_{\Phi_{\bf\gamma}(p_\ell)}={\bf e}^t_kM_{\bf\gamma}{\bf e}_\ell.
\end{equation}

When $m=1$, the claim follows from the definition of $M_\gamma$.

For $m\geq2$, we write ${\bf\gamma}=\gamma_1\gamma_2\cdots\gamma_m$ and denote ${\bf\gamma}^-=\gamma_1\gamma_2\cdots\gamma_{m-1}$. 
By checking the boundary conditions, it is not hard to verify
\begin{equation}
v|_{\Omega_{\bf\gamma^-}\setminus \Phi_{\bf\gamma^-}(V_0)}=\sum_{s:p_s\in\Omega_{T({\bf\gamma}^-)}} \left(d(v|_{\Omega_{{\bf\gamma}^-}})\right)_{\Phi_{{\bf\gamma}^{-}}(p_s)}\varphi_{s}\circ \Phi^{-1}_{\gamma^-},
\end{equation}
noticing that both sides are harmonic in $\Omega_{\bf\gamma^-}\setminus \Phi_{\bf\gamma^-}(V_0)$, where
\begin{equation*}
\varphi_{s}=\begin{cases}R_{\Omega_{{\bf\gamma}^-}}(\partial\Omega_{{\bf\gamma}^-},\Phi_{{\bf\gamma}^{-}}(p_s))\cdot H^{\{\{p_{s}\},\partial\Omega_{T({\bf\gamma}^-)}\}}({\bf1}_{p_s})\quad &\text{if $p_s\in\Omega_{T({\bf\gamma}^-)},$}\\
0  &\text{if $p_s\notin\Omega_{T({\bf\gamma}^-)}.$}\end{cases}
\end{equation*}
Hence we have
\begin{equation*}
\left(d(v|_{\Omega_{\bf\gamma}})\right)_{\Phi_{\bf\gamma}(p_\ell)}=\sum_{s:p_s\in\Omega_{T({\bf\gamma}^-)}}  \left(d(v|_{\Omega_{{\bf\gamma}^-}})\right)_{\Phi_{{\bf\gamma}^{-}}(p_s)}M_{\gamma_m}(s,\ell),
\end{equation*}
which by induction yields that
\begin{equation*}
\left(d(v|_{\Omega_{\bf\gamma}})\right)_{\Phi_{\bf\gamma}(p_\ell)}=\sum_{s=1}^Q {\bf e}^t_kM_{\bf\gamma^-}{\bf e}_s M_{\gamma_m}(s,\ell)={\bf e}^t_kM_{\bf\gamma}{\bf e}_{\ell},
\end{equation*}
proving the claim.

By the definition of $\mu_{i,k}$ and using the claim, we obtain
\begin{equation*}
\mu_{i,k}(\partial\Omega_{\bf\gamma})=\sum_{\ell=1}^Q{\bf e}^t_kM_{\bf \gamma}{\bf e}_{\ell}=\sum_{\ell:p_\ell\in\Omega_{T({\bf\gamma})}}\left(d(v|_{\Omega_{\bf\gamma}})\right)_{\Phi_{\bf\gamma}(p_\ell)}=-(dv)_{\partial\Omega_{\bf\gamma}},
\end{equation*}
as desired, where the last equality is from \eqref{zerodiv}.
\end{proof}

\begin{proposition}\label{thmpm}
 For $p_k\in \Omega_i\cap V_0$, $\mu_{i,k}$ uniquely extends to a Borel probability measure on $\partial\Omega_{i}$. Moreover, we have the identity
 $$\mu_{i,k}=\sum_{\gamma\in\Gamma(i),1\leq\ell\leq Q}M_{\gamma}(k,\ell)\mu_{T(\gamma),\ell}\circ\theta_{\gamma}^{-1}.$$
\end{proposition}
\begin{proof}
By Proposition \ref{BGDprop}(i), for ${\bf\gamma}\in\Gamma_m(i)$, $m\geq1$, it holds that $\Phi_{\bf\gamma}(\Omega_{T({\bf\gamma})}\cap V_0)\neq\emptyset$.
On the other hand, by \eqref{eqeq3.1},
\begin{equation*}
\sum_{\eta\in \Gamma(i)}\mu_{i,k}(\partial\Omega_{\eta})=\sum_{\eta\in \Gamma(i)}{\bf e}^t_kM_{\eta}{\bf 1}=\sum_{\eta\in\Gamma({i})}\sum_{\ell=1}^QM_{\eta}(k,\ell)=1,
\end{equation*}
and similarly for any ${\bf \gamma}\in \Gamma_*$ with $I({\bf\gamma})=i$,
\begin{equation*}
\sum_{\eta\in \Gamma(T(\bf\gamma))}\mu_{i,k}(\partial\Omega_{{\bf \gamma}\eta})={\bf e}^t_kM_{\bf\gamma}\sum_{\eta\in \Gamma(T(\bf\gamma))}M_{\eta}{\bf 1}={\bf e}^t_kM_{\bf\gamma}{\bf 1}=\mu_{i,k}(\partial\Omega_{\bf \gamma}).
\end{equation*}
Hence $\mu_{i,k}$ can be uniquely extended to a Borel probability measure on $\partial\Omega_i$ by the Kolmogorov extension theorem.

Moreover, for ${\bf\gamma}=\gamma_1\cdots\gamma_m\in\Gamma_*$, we have
\begin{equation}\label{eqmeasuressi}
\mu_{i,k}(\partial\Omega_{\bf \gamma})={\bf e}^t_kM_{\gamma_1}M_{\gamma_2\cdots\gamma_m}{\bf 1}=\sum_{\ell=1}^QM_{\gamma_1}(k,\ell){\bf e}^t_\ell M_{\gamma_2\cdots\gamma_m}{\bf 1}=\sum_{\ell=1}^QM_{\gamma_1}(k,\ell)\mu_{T(\gamma_1),\ell}(\theta_{\gamma_1}^{-1}(\partial\Omega_{\bf\gamma})).
\end{equation}
Summing up \eqref{eqmeasuressi} with $\gamma_1\in\Gamma(i)$, we obtain
\begin{equation*}
\mu_{i,k}=\sum_{\gamma\in\Gamma(i)}\sum_{\ell=1}^QM_{\gamma}(k,\ell)\mu_{T(\gamma),\ell}\circ\theta_{\gamma}^{-1},
\end{equation*}
which finishes the proof.
\end{proof}

We then prove that the probability measures $\{\mu_{i,k}:\ p_k\in \Omega_i\cap V_0,\ 1\leq k\leq Q\}$ are exactly the hitting probabilities associated with $\Omega_i$,  $1\leq i\leq P$. This is the main result in this section.

\begin{theorem}\label{th3.1}
For $p_k\in\Omega_i\cap V_0$, the probability measure $\mu_{i,k}$ in Definition \ref{defhmeasure} is the hitting probability from $p_k$ to the $R$-boundary $\partial\Omega_i$. Consequently, for any $f\in C(\partial\Omega_i)$, the unique harmonic function $u$ on $\Omega_i$ generated by $f$, i.e. $u|_{\partial\Omega_i}=f$, satisfies
\begin{equation}\label{Poisson}
u(p_k)=\int_{\partial\Omega_i}f(x)d\mu_{i,k}(x).
\end{equation}
\end{theorem}
\begin{proof}
We prove the result only when $f$ is a simple function on $\partial\Omega_i$, since the general case will follow by approximating with simple functions.
Let $m\geq1$ be an integer, assume that $f$ is of the form
\begin{equation}\label{expf}
f=\sum_{{\bf \gamma}\in\Gamma_m(i)}f_{\bf\gamma}1_{\partial\Omega_{\bf \gamma}},\qquad f_{\bf \gamma}\in\mathbb{R}.
\end{equation}
Then $f$ is continuous on $\partial \Omega_i$.
Let $u$ be the unique harmonic extension of $f$ on $\Omega_i$.
Let $v=R_{\Omega_i}(\partial\Omega_i,p_k)\varphi$, with $\varphi$ being the realization of $R_{\Omega_i}(\partial\Omega_i,p_k)$ as above.
Notice that both $u$ and $v$ are harmonic in $\Omega_i\setminus\{p_k\}$ and take finitely many different values on the boundary $\partial\Omega_i$. 
The trace of the energy $\mathcal E_{\Omega_i}$ to $\{p_k\}\cup\{\partial\Omega_{\bf \gamma}:\ {\bf \gamma}\in\Gamma_m(i)\}$ is an electric network, and thus we can apply \eqref{DisGuass-Green} with $u$ and $v$ to obtain
\begin{equation*}
\sum_{{\bf \gamma}\in\Gamma_m(i)}u(\partial\Omega_{\bf\gamma})(dv)_{\partial\Omega_{\bf\gamma}}+u(p_k)(dv)_{p_k}=\sum_{{\bf \gamma}\in\Gamma_m(i)}v(\partial\Omega_{\bf\gamma})(du)_{\partial\Omega_{\bf\gamma}}+v(p_k)(du)_{p_k}=0,
\end{equation*}
where in the last equality we use that $v=0$ on $\partial\Omega_i$ and $(du)_{p_k}=0$ by the harmonicity of $u$ at $p_k$. Then by $(dv)_{p_k}=1$ and \eqref{expf}, we obtain from above that
\begin{equation}\label{eqforf}
\sum_{{\bf \gamma}\in\Gamma_m(i)}f_{\bf\gamma}(dv)_{\partial\Omega_{\bf\gamma}}+u(p_k)=0.
\end{equation}
By Lemma \ref{lemma5.4} and \eqref{eqforf}, we obtain
\begin{equation*}
u(p_k)=\sum_{{\bf \gamma}\in\Gamma_m(i)}f_{\bf\gamma}\mu_{i,k}(\partial\Omega_{\bf\gamma}),
\end{equation*}
proving that \eqref{Poisson} holds for any simple function $f$.
\end{proof}

{\it \noindent{\bf Remark 1.} Recall the first remark after Theorem \ref{thm3.1},  a function $f\in C(D_i)$ naturally induces a function $f\circ\iota_i\in C(\partial\Omega_i)$. In this way, the harmonic function generated by $f\circ\iota_i$ can be viewed as a harmonic extension of $f$ from $D_i$ to $\Omega_i$.

\medskip

\noindent{\bf Remark 2.} In Theorem \ref{th3.1}, if $u$ is harmonic in $\Omega_i\setminus V_0$, similarly to the proof of \eqref{Poisson}, for $p_k\in\Omega_i\cap V_0$, we have
\begin{equation}\label{Poissonex}
u(p_k)=\int_{\partial\Omega_i}f(x)d\mu_{i,k}(x)+\sum_{x\in \Omega_i\cap V_0}v(x)(du)_x,
\end{equation}
where $v=R_{\Omega_i}(\partial\Omega_i,p_k)\varphi$ with $\varphi$ being the realization of $R_{\Omega_i}(\partial\Omega_i,p_k)$.
}

\medskip

The following property says that for fixed $i$, the measures $\mu_{i,k}$ are actually equivalent for different $p_k\in\Omega_i$. For convenience, we will also write the measure $\mu_{i,k}$ as $\mu_{i,p}$ if we denote $p_k$ by $p$.
\begin{proposition}\label{thmeqmea}
Let $i\in\mathcal A$, assume $p,p'\in \Omega_i\cap V_0$ and let $\mu_{i,p},\mu_{i,p'}$ be the associated probability measures. Then there exists a constant $C>0$ such that for any Borel set $E\subset \partial\Omega_i$,
\begin{equation}\label{eqmea}
C^{-1}\mu_{i,p}(E)\leq \mu_{i,p'}(E)\leq C\mu_{i,p}(E).
\end{equation}
\end{proposition}
\begin{proof}
Without loss of generality, we may assume $E=\partial\Omega_{\bf\gamma}$ for some ${\bf\gamma}\in\Gamma_m(i)$, $m\geq1$. Let $u_{\bf\gamma}$ be the harmonic function in $\Omega_i$ with boundary values
\begin{equation*}
u_{\bf\gamma}=\begin{cases}
& 1\quad\text{ in $\partial\Omega_{\bf\gamma}$,}\\
& 0\quad\text{ in $\partial\Omega_i\setminus \partial\Omega_{\bf\gamma}$}.
\end{cases}
\end{equation*}
Then $u_{\bf\gamma}>0$ in $\Omega_i$ by the strong maximum principle (e.g. \cite[Theorem 4.1]{Kuw}). Note that by Theorem \ref{th3.1}, $\mu_{i,p}(E)=u_{\bf\gamma}(p)$ and $\mu_{i,p'}(E)=u_{\bf\gamma}(p')$. Then by Lemma \ref{connectedlemma}, we pick $n_1\geq1$ such that $p$ and $p'$ are connected by a chain of $n_1$-cells in $\Omega_i$ and we denote the union of these cells by $A$. Now we consider
\begin{equation*}
A'=\bigcup_{\omega\in\Sigma^{n_1},F_{\omega}(K)\cap A\neq\emptyset}F_{\omega}(K)
\end{equation*}
and may assume that $A'\subset \Omega_i$ by choosing $n_1$ sufficiently large. The set
\begin{equation*}
\bigcup_{\omega\in\Sigma^{n_1},F_{\omega}(K)\cap A\neq\emptyset}F_{\omega}(V_0)\setminus A
\end{equation*}
 is a nonempty finite set in $V_{n_1}$, denoted as $\{q_1,\ldots,q_{\ell}\}$. Then $A'\setminus\{q_1,\ldots,q_\ell\}$ is connected.
Since $u_{\bf\gamma}$ is harmonic in $\Omega_i$, by viewing $\{q_1,\ldots,q_{\ell}\}$ as the boundary of $A'$, we see that there is a positive probability vector $(w_{1},\ldots,w_{\ell})$ such that
\begin{equation}\label{measurek}
u_{\bf\gamma}(p)=\sum_{s=1}^{\ell}w_{s}u_{\bf\gamma}(q_s),
\end{equation}
 where $\sum_{s=1}^{\ell}w_{s}=1$ and $w_{s}>0$ depending only on the resistance form and $A'$. Similarly, there is a positive probability vector $(w'_{1},\ldots,w'_{\ell})$ such that
\begin{equation}\label{measurek'}
 u_{\bf\gamma}(p')=\sum_{s=1}^{\ell}w'_{s}u_{\bf\gamma}(q_s).
\end{equation}
  Now since $q_s\in \Omega_i$, we have $u_{\bf\gamma}(q_s)>0$. Combining \eqref{measurek} and \eqref{measurek'}, we see at
\begin{equation*}
\min_{1\leq s\leq \ell}\frac{w_{s}}{w'_{s}}\leq \frac{u_{\bf\gamma}(p)}{u_{\bf\gamma}(p')}\leq\max_{1\leq s\leq \ell}\frac{w_{s}}{w'_{s}},
\end{equation*}
which implies \eqref{eqmea}.
\end{proof}

\section{\bf Energy estimates}\label{sec6}
\setcounter{equation}{0}\setcounter{theorem}{0}

In this section, we characterize harmonic functions in $\Omega$ with finite energy in terms of their boundary values.

Let $\{(\Omega_i,D_i)\}_{i=1}^P$ be domains in a p.c.f. self-similar set {$(K,\{F_i\}_{i=1}^N,V_0)$} satisfying the BGD condition (\ref{graphdirected}) and $(\mathcal E,\mathcal F)$ be a self-similar resistance form with energy renormalizing factors $\{r_i\}_{i=1}^N$, $0<r_i<1$.

Before proceeding, we give a property of the energies of harmonic functions. 
\begin{lemma}\label{lemma6.1}
Let $u$ be a harmonic function in $K\setminus V_0$. We have
\begin{equation*}
\mathcal E[u]\asymp\sum_{p,q\in V_0}|u(p)-u(q)|^2\asymp\sum_{p\in V_0}\left|(du)_p\right|^2,
\end{equation*}
where the positive constants in the two ``$\asymp$"'s are independent of $u$.
\end{lemma}
\begin{proof}
The lemma follows from the fact that both of the last two terms are (square of) norms on $\ell(V_0)$ modulo constants.
\end{proof}


Let $i\in\{1,\ldots,P\}$. For a harmonic function $u$ on $\Omega_i$ with boundary value $f$ on $\partial\Omega_i$,
our purpose is to estimate $\mathcal E_{\Omega_i}[u]$ from above and below in terms of $f$.

For two words $\mathbf{\eta},\mathbf{\xi}\in \Gamma_m(i)$ with $m\geq1$, we write ${\bf\eta}\sim{\bf\xi}$ if $[{\bf\eta}]_{m-1}=[{\bf\xi}]_{m-1}={\bf\gamma}$ for some ${\bf\gamma}\in\Gamma_{m-1}(i)$ (we also write ${\bf\eta}^{-}={\bf\xi}^{-}={\bf\gamma}$). Note that the
two $m$-cells $\Phi_{\bf\eta}(K)$ and $\Phi_{\bf\xi}(K)$ are contained in the same $(m-1)$-cell $\Phi_{\bf\gamma}(K)$. 
In the following, we denote $V^{(\bf\gamma)}:=\ V_0\cap \Omega_{T({\bf \gamma})}$ and $r_{\bf\gamma}=r_{\omega}$ with the unique $\omega\in\Sigma^{m-1}$ satisfying $F_{\omega}=\Phi_{\bf\gamma}$.

For $f\in C(\partial\Omega_i)$ and $p\in V^{(\bf\gamma)}$, we denote
\begin{equation*}
f_{{\bf\gamma},p}=\int_{\partial\Omega_{T({\bf\gamma})}}f\circ\theta_{\bf\gamma}d\mu_{T({\bf\gamma}),p}.
\end{equation*}
Our main result in this section is the following.
\begin{theorem}\label{enes}
Let $(K,\{F_i\}_{i=1}^N,V_0)$ be a p.c.f. self-similar set equipped with a self-similar resistance form $(\mathcal E,\mathcal F)$ with energy renormalizing factors $\{r_i\}_{i=1}^N$, $0<r_i<1$. Let $\{(\Omega_i,D_i)\}_{i=1}^P$ be a finite number of domains in $K$ satisfying the BGD condition. Assume $i\in\{1,\ldots,P\}$ and $f\in C(\partial\Omega_i)$, let $u$ be the unique harmonic extension of $f$ in $\Omega_i$. Then
\begin{equation}\label{energyestimate}
\mathcal E_{\Omega_i}[u]\asymp \sum_{m=0}^{\infty}\sum_{{\bf\gamma}\in\Gamma_m(i)}\frac1{r_{\bf\gamma}}\sum_{{\bf\xi},{\bf\eta}:\ {\bf\xi}^{-}={\bf\eta}^{-}={\bf\gamma}}\sum_{p\in V^{({\bf\xi})},\ q\in V^{{(\bf\eta)}}}\left(f_{{\bf\xi},p}-f_{{\bf\eta},q}\right)^2,
\end{equation}
where the constant in ``$\asymp$" does not depend on $u$ or $f$.
\end{theorem}







\begin{proof}
We first show ``$\lesssim$" in \eqref{energyestimate}. For a given $f\in C(\partial\Omega_i)$, we will construct a continuous function $h$ in $\widetilde\Omega_i$ such that $h=f$ on $\partial\Omega_i$.

For $m\geq1$, let $\mathcal W_m=\{\omega\in \Sigma^m :\ F_{\omega}(K)\subset \Omega_i, F_{\omega^-}(K)\nsubset\Omega_i\}$ ($\omega^{-}=[\omega]_{m-1}$) and $\mathcal W_*=\bigcup_{m=1}^\infty \mathcal W_m$.
Define $\Omega^m=\bigcup\{F_{\omega}(K): \omega\in \mathcal W_m \}$ and for ${\bf\gamma}\in \Gamma_{m-1}(i)$, write $\Omega^{m}_{\bf\gamma}=\Omega^m\cap \Omega_{\bf\gamma}$. Clearly, $\Omega_i=\bigcup_{m=1}^{\infty}\Omega^m$ and $\Omega^m=\bigcup_{{\bf\gamma}\in\Gamma_{m-1}(i)}\Omega^m_{\bf\gamma}$. Define $T_m=\{\Phi_{\bf\gamma}(V^{({\bf\gamma})}):\ {\bf\gamma}\in\Gamma_m(i)\}$ for $m\geq1$ and $T_0=\emptyset$. Write $T_*=\bigcup_{m=1}^\infty T_m$.

\medskip

{\it Claim. (1). $\{\Omega^m_{\bf\gamma}\setminus T_*\}$ are disjoint open subsets in $\Omega_i$;

(2). there exists an integer $m_0\geq1$ such that for all $m\geq1$, $T_m\cap \bigcup_{k\geq m+m_0}T_{k}=\emptyset$, $\Omega^m\cap \bigcup_{k\geq m+m_0}\Omega^{k}=\emptyset$;

(3). the boundary of each $\Omega^m_{\bf\gamma}\setminus T_*$, denoted by $\partial\Omega^m_{\bf\gamma}$, is contained in $T_{m-1}\cup T_m$ and has no intersection with $T_k$ for $k\leq m-1-m_0$ or $k\geq m+m_0$.
}

\medskip

Indeed, (1) is obvious. For (2), by the BGD condition, for $\omega\in\mathcal W_m$, $F_{\omega^{-}}(K)\cap \Omega_i=\Omega_{\bf\gamma}$ for some ${\bf\gamma}\in\Gamma_{m-1}(i)$. Hence $F_\omega(K)$ is a closed subset in $\Omega_{\bf\gamma}$ and has positive distance to its geometric boundary $D_{\bf\gamma}$, so that $F_\omega(K)\cap \bigcup_{{\bf\eta}\in \Gamma_{m+m_1}(i):[{\bf\eta}]_{m}={\bf\gamma}}\Omega_{\bf\eta}=\emptyset$ if we pick $m_1$ sufficiently large. Since the choice of $\Omega_{T(\bf\gamma)}$ is finite, we can pick a common $m_1$ (independent of $m$). Also we can pick $m_2$ large such that for any $j\in\{1,\ldots,P\}$ and $\gamma\in \Gamma(j)$, $\Omega_j\cap V_0\cap \bigcup_{{\bf\xi}\in\Gamma_{m_2}(j)}\Omega_{\bf\xi}=\emptyset$.  Then $m_0:=\max\{m_1,m_2\}$ is so that (2) holds. (3) follows from (2).

For $m\geq1$ and for each point $x=\Phi_{\bf\gamma}(p)\in T_m\setminus \left(\bigcup_{j=0}^{m-1}T_{j}\right)$ with $p\in V^{(\bf\gamma)}$, ${\bf\gamma}\in\Gamma_m(i)$, define $h(x)=f_{{\bf\gamma},p}$. When more than one $({\bf\gamma},p)$ satisfies $x=\Phi_{\bf\gamma}(p)$, we arbitrarily choose one of them to define $h(x)$.
We then do harmonic extension of $h$ in each $\Omega^m_{\bf\gamma}\setminus T_*$ through its boundary values which are well defined already. 

 By the construction, for $\Omega^m_{\bf\gamma}$ with $m> m_0$ sufficiently large (${\bf\gamma}\in\Gamma_{m-1}(i)$), note that the values of $h$ in each $\Omega_{\bf\gamma}^m\setminus T_*$ are determined by its values on $\partial\Omega^m_{\bf\gamma}\subset \Phi_{\bf\gamma}(V^{({\bf\gamma})})\cup \bigcup_{{\bf\eta}:{\bf\eta}^-={\bf\gamma}}\Phi_{\bf\eta}(V^{({\bf\eta})})$. From this, by the Claim, we see that for each $x\in \partial\Omega^m_{\bf\gamma}$, if it has an expression $x=\Phi_{{\bf\gamma'}}(p)$ for some ${\bf\gamma}'\in\Gamma_k(i)$ and $p\in V^{({\bf\gamma'})}$, we must have $k\geq m-m_0$ and $[{\bf\gamma}']_{m-m_0-1}=[{\bf\gamma}]_{m-m_0-1}$. So the value $h(x)$ lies in 
\begin{equation}\label{eqadsz}
\big\{f_{{\bf\gamma}',p}: {\bf\gamma}'\in \Gamma_k(i), m-m_0-1\leq k\leq m,[{\bf\gamma}']_{m-m_0-1}=[{\bf\gamma}]_{m-m_0-1}, p\in \Phi_{\bf\gamma'}(V^{({\bf\gamma'})})\big\}.
\end{equation}
 Thus the values of $h$ in $\Omega^m_{\bf\gamma}$ is defined by integrating the values of $f$ on $\partial\Omega_{[{\bf\gamma}]_{m-m_0-1}}$ against probability measures. Since $f$ is continuous, we see that $h$ is continuous on $\widetilde\Omega_i$.

Now we estimate $\mathcal E_{\Omega^m}[h]$ for $m\geq1$. By using  \eqref{eqadsz}, Lemma \ref{lemma6.1} and the triangle inequality, we have
\begin{align}
&\mathcal{E}_{\Omega^m}[h]=\sum_{{\bf\gamma}\in\Gamma_{m-1}(i)}\mathcal{E}_{\Omega_{\bf\gamma}^m}[h]\lesssim\sum_{{\bf\gamma}\in\Gamma_{m-1}(i)}\frac1{r_{\bf\gamma}}\sum_{x,y\in \partial\Omega^m_{\bf\gamma}}(h(x)-h(y))^2\notag\\
\lesssim&\sum_{k=\max\{m-m_0-1,0\}}^{m-1}\sum_{{\bf\gamma}\in\Gamma_{k}(i)}\frac1{r_{\bf\gamma}}\left(\sum_{{\bf\xi}^-={\bf\eta}^-={\bf\gamma}}\sum_{p\in V^{({\bf\xi})},\ q\in V^{({\bf\eta})}}\left(f_{{\bf\xi},p}-f_{{\bf\eta},q}\right)^2
+\sum_{{\bf\eta}^-={\bf\gamma}}\sum_{p\in V^{({\bf\gamma})},\ q\in V^{({\bf\eta})}}\left(f_{{\bf\gamma},p}-f_{{\bf\eta},q}\right)^2\right).
\label{the2term}
\end{align}
Observe that by Proposition \ref{thmpm}, for each ${\bf\gamma}$, $f_{{\bf\gamma},p}$ is a linear combination of those $f_{{\bf\eta},q}, {\bf\eta}^-={\bf\gamma}$ with probability weights, and the weights are some constants independent of $f$. Hence the second term in the summation on the RHS of \eqref{the2term} can be absorbed into the first term. We obtain
\begin{align}\label{eq6.23}
\mathcal{E}_{\Omega^m}[h]&\lesssim\sum_{k=\max\{m-m_0,0\}}^{m-1}\sum_{{\bf\gamma}\in\Gamma_{k}(i)}\frac1{r_{\bf\gamma}}\sum_{{\bf\xi}^-={\bf\eta}^-={\bf\gamma}}\sum_{p\in V^{({\bf\xi})},\ q\in V^{({\bf\eta})}}\left(f_{{\bf\xi},p}-f_{{\bf\eta},q}\right)^2.
\end{align}
By summing up the estimate \eqref{eq6.23} over all $m\geq1$, we have
\begin{align*}\label{eq5.18}
\mathcal E_{\Omega_i}[u]\leq \mathcal E_{\Omega_i}[h]=\sum_{m=1}^\infty\mathcal E_{\Omega^m}[h]\lesssim\sum_{m=0}^{\infty}\sum_{{\bf\gamma}\in\Gamma_m(i)}^{\infty}\frac1{r_{\bf\gamma}}\sum_{{\bf\xi}^{-}={\bf\eta}^{-}={\bf\gamma}}\sum_{p\in V^{({\bf\xi})},\ q\in V^{({\bf\eta})}}\left(f_{{\bf\xi},p}-f_{{\bf\eta},q}\right)^2,
\end{align*}
which proves ``$\lesssim$" in \eqref{energyestimate}
.




\medskip

We then prove the ``$\gtrsim$" of \eqref{energyestimate}. For ${\bf \xi}\in\Gamma_m(i)$ with $m\geq1$, we denote $u^{({\bf\xi})}:=u\circ \Phi_{\bf\xi}$, a function on $\Omega_{T({\bf\xi})}$. For $p\in V^{(\bf\xi)}$, denote $v^{({\bf\xi},p)}=R_{\Omega_{T({\bf\xi})}}(\partial\Omega_{T({\bf\xi})},p)\varphi$ with $\varphi$ being the realization of $R_{\Omega_{T({\bf\xi})}}(\partial\Omega_{T({\bf\xi})},p)$. We apply \eqref{Poissonex} to obtain
\begin{equation}\label{eq6.24}
f_{{\bf\xi},p}=u^{({\bf\xi})}(p)-\sum_{q\in V^{({\bf\xi})}}v^{({\bf\xi},p)}(q)(du^{({\bf\xi})})_q.
\end{equation}
Now for a pair ${\bf\xi}\sim{\bf\eta}$ in $\Gamma_m(i)$ (i.e. ${\bf\xi}^-={\bf\eta}^-={\bf\gamma}$ for some ${\bf\gamma}\in\Gamma_{m-1}(i)$) and $p\in V^{(\bf\xi)}$, $p'\in V^{(\bf\eta)}$, noticing that $0\leq v^{({\bf\xi},p)}\leq R_{\Omega_{T({\bf\xi})}}(\partial\Omega_{T({\bf\xi})},p)$ and $0\leq v^{({\bf\eta},p')}\leq R_{\Omega_{T({\bf\eta})}}(\partial\Omega_{T({\bf\eta})},p')$, which are bounded by a universal constant, we obtain
\begin{align}\label{eq522}
\left(f_{{\bf\xi},p}-f_{{\bf\eta},p'}\right)^2
&\lesssim
\sum_{q\in V^{({\bf\xi})}}\left|(du^{({\bf\xi})})_q\right|^2+\sum_{q\in V^{({\bf\eta})}}\left|(du^{({\bf\eta})})_q\right|^2+(u^{({\bf\xi})}(p)-u^{({\bf\eta})}(p'))^2\notag\\
&\lesssim \begin{cases} r_{\bf\gamma}\mathcal E_{\left(\bigcup_{k=m}^{m+n_1}\Omega^k\right)\cap \Omega_{\bf\gamma}}[u],\qquad &\text{ if }{\bf\xi}\neq{\bf\eta},\\
r_{\bf\xi}\mathcal E_{\left(\bigcup_{k=m}^{m+n_1}\Omega^k\right)\cap \Omega_{\bf\xi}}[u],\ \ \quad &\text{ if }{\bf\xi}={\bf\eta},
\end{cases}
\end{align}
where $n_1$ is the same as that in Lemma \ref{connectedlemma}, and we have used Lemma \ref{lemma6.1} and triangle inequality in \eqref{eq522}. Summing up (\ref{eq522}) over all pairs ${\bf\xi}\sim{{\bf\eta}}$ in $\Gamma_m(i)$ and all possible $p,p'$, we get
\begin{equation}\label{eq6.26}
\frac1{r_{\bf\gamma}}\sum_{{\bf\gamma}\in\Gamma_{m-1}(i)}\sum_{{\bf\xi}^{-}={\bf\eta}^{-}={\bf\gamma}}\sum_{p\in V^{({\bf\xi})},\ q\in V^{({\bf\eta})}}\left(f_{{\bf\xi},p}-f_{{\bf\eta},q}\right)^2
\lesssim \sum_{k=m}^{m+n_1}\mathcal E_{\Omega^k}[u].
\end{equation}
Summing up the inequalities \eqref{eq6.26} over all $m\geq1$, we finally obtain
\begin{equation*}
\sum_{m=0}^{\infty}\sum_{{\bf\gamma}\in\Gamma_m(i)}^{\infty}\frac1{r_{\bf\gamma}}\sum_{{\bf\xi}^{-}={\bf\eta}^{-}={\bf\gamma}}\sum_{p\in V^{({\bf\xi})},\ q\in V^{({\bf\eta})}}\left(f_{{\bf\xi},p}-f_{{\bf\eta},q}\right)^2
\lesssim\mathcal E_{\Omega_i}[u],
\end{equation*}
 proving ``$\gtrsim$" in \eqref{energyestimate}.


\end{proof}

\section{\bf Examples}\label{sec7}
\setcounter{equation}{0}\setcounter{theorem}{0}

In this section, we present several examples. We will first consider the Sierpinski gasket (SG) as a typical example. There is a large class of domains in SG which are constructed by using a straight line to ``cut" the SG. We prove that these domains will satisfy the BGD condition if the line is passing through two points in $V_*$ of SG. Then for some typical cases in this class, we compute the corresponding flux transfer matrices which generate the hitting probability measures, see \cite{OS,GKQS,LS,CQ,KT22} for several previous works. We also present some other examples satisfying the BGD as well as some calculations.

\subsection{Example: Sierpinski gasket}\label{subsec7.1}
Let $K$ be the Sierpinski gasket in $\mathbb R^2$, generated by the IFS $\{F_i\}_{i=1}^3$ with $F_i(x)=\frac 12 (x-p_i)+p_i$, $i=1,2,3$, and $V_0=\{p_1,p_2,p_3\}$ is the three vertices of an equilateral triangle $T$. The standard resistance form $(\mathcal{E},\mathcal{F})$ on $K$  satisfies the self-similar identity \cite{K1}
\begin{equation*}
\mathcal E[u]=\frac53\sum_{i=1}^3\mathcal E[u\circ F_{i}], \quad \forall u\in \mathcal{F}.
\end{equation*}

Let $L\subset \mathbb R^2$ be the straight line. Then $L$ separates the plane into two disjoint (open) parts, say $H_1$ and $H_2$. Denote $\Omega_1=K\cap H_1$ and $\Omega_2= K\cap H_2$.
\begin{proposition}\label{thm7.2}
Each of the two sets $\Omega_1$ and $\Omega_2$ is arcwise connected. Moreover, for $i=1$ or $2$, assume $\Omega_i\neq\emptyset$, then unless $L$ includes an edge of $F_{\omega}(T)$ for some $\omega\in \cup_{n\geq0}\Sigma^n$ and $\Omega_i$ contains two points in $V_0$, the geometric boundary of $\Omega_i$ is $L\cap K$.
\end{proposition}
\begin{proof}
For simplicity, we denote $H=H_i$ and $\Omega=\Omega_i$ for $i=1$ or $2$ and assume $\Omega\neq\emptyset$. Note that $\#(\Omega\cap V_0)=1$ or $2$.

 First we prove that $\Omega$ is arcwise connected. For $x\in K$, denote by $d(x,L)=\inf\{|x-y|: y\in L\}$ the distance from $x$ to $L$. Then by the geometry of $K$, $d=\max\{d(x,L):x\in\Omega\}$ is attained by some point in $V_0\cap\Omega$, say $p_1$ for simplicity.

We claim that

\medskip

 {\it for any $x_0\in \Omega$, there is a curve included in $\Omega$ connecting $x_0$ and $p_1$}.

\medskip

 Indeed, by that $\Omega$ is open, there is a word $\omega^{(1)}\in\Sigma^n$ for some $n\geq1$ such that $x_0\in F_{\omega^{(1)}}(K)\subset \Omega$. Clearly, there exists a curve inside $F_{\omega^{(1)}}(K)$ connecting $x_0$ to each point in $F_{\omega^{(1)}}(V_0)$. Let $d_1=\max\{d(x,L): x\in F_{\omega^{(1)}}(V_0)\}$ and $x_1\in F_{\omega^{(1)}}(V_0)$ be such that $d(x_1,L)=d_1$. If $d_1<d$, we can find another word $\omega^{(2)}\in \Sigma^n$ such that $F_{\omega^{(1)}}(K)\cap F_{\omega^{(2)}}(K)= \{x_1\}$, $F_{\omega^{(2)}}(K)\subset\Omega$ and $d_2:=\max\{d(x,L):\ x\in F_{\omega^{(2)}}(V_0)\}>d_1$. If $d_2<d$, we do the same thing with $\omega^{(2)}$ and continue this procedure to find a finite chain of $n$-cells $\{F_{\omega^{(i)}}(K)\}^m_{i=1}$ and a finite sequence of points $\{x_i\}^m_{i=1}$ such that $d_i:=d(x_i,L)$ is strictly increasing for $i$ and $d_m=d$. Note that the only possible case $x_m\neq p_1$ happens when $\#(\Omega\cap V_0)=2$ and $L$ is parallel to the line passing through $\Omega\cap V_0$. In this case, the line segment joining $x_m$ and $p_1$ is obviously included in $\Omega$.
 From the construction, for $i=1,\ldots,m$ there is a curve included in $F_{\omega^{(i)}}(K)(\subset\Omega)$ connecting $x_{i-1}$ and $x_{i}$. Hence by concatenating these curves, we obtain a curve included in $\Omega$ connecting $x_0$ and $p_1$. The claim holds and hence $\Omega$ is arcwise connected.

\medskip

Then we prove the second assertion. Denote by $D$ the geometric boundary of $\Omega$. It is clear that $D\subset L\cap K$.

 First, assume that for any $\tau\in \cup_{n\geq0}\Sigma^n$, $L$ does not include any edge of $F_{\tau}(T)$, thus $L$ contains at most one point in $F_{\tau}(V_0)$. For any $x\in L\cap K$, let $\{\tau^{(n)}\}_{n=1}^\infty$ be a sequence of words such that for each $n \geq1$, $\tau^{(n)}\in\Sigma^n$, $x\in F_{\tau^{(n)}}(K)$ and $F_{\tau^{(n)}}(V_0)\cap \Omega\neq\emptyset$. Hence for any $n\geq1$, we can pick $x_n\in F_{\tau^{(n)}}(V_0)\cap \Omega$ so that $x_n\rightarrow x$ as $n\rightarrow\infty$, which implies that $x\in D$ and thus $L\cap K\subset D$.

Second, assume $L$ includes an edge of $F_{\omega}(T)$ for some word $\omega\in \Sigma^n$, $n\geq0$. In this case, $L$ is parallel to one edge (denoted by $S$) of $T$ and $L\cap K=\bigcup_{k=1}^mF_{\omega^{(k)}}(S)$ for a finite number of $n$-words $\{\omega^{(k)}\}_{k=1}^m$. If $\Omega$ contains only one point in $V_0$, then $F_{\omega^{(k)}}(K)\setminus L\subset\Omega$ for each $1\leq k\leq m$, and hence $L\cap K\subset D$; otherwise, $\Omega$ contains two points in $V_0$, then it is immediate from the geometry of $\Omega$, $D=\bigcup_{k=1}^mL\cap F_{\omega^{(m)}}(V_0)$, which is finite and strictly contained in $L\cap K$.

Combing the above two cases, we conclude that the second assertion holds.
\end{proof}

In the following, pick arbitrarily two distinct points $p,q\in V_*$, assume $L$ is the straight line passing through $p$ and $q$.
\begin{proposition}\label{thm7.2}
Let $L,p,q, \Omega_1,\Omega_2$ be as above.
For $i=1$ or $2$, if $\Omega_i\neq\emptyset$, then $\Omega_i$ satisfies the BGD condition.
\end{proposition}
\begin{proof}
For convenience, we write ${\bf e}_1=\overrightarrow{p_1p_2}$, ${\bf e}_2=\overrightarrow{p_1p_3}$ for two unit vectors, where $p_1=O=(0,0)$, $p_2=(1,0)$ and $p_3=(\frac12,\frac{\sqrt3}2)$. Denote $I$ the line segment joining $p_1$ and $p_2$.
 By symmetry, we may assume that the line $L$ has the slope in $[0,\sqrt{3})$.

 If the slope of $L$ is $0$, then we assume $L$ passes through $\frac k{2^n}(\frac12,\frac{\sqrt3}2)$ for two integers $n\geq 0$ and $k\in[0,2^n]$. We may also require that $k$ is either $0$ or an odd number. If $k=0$, then $L$ is the line passing through $p_1,p_2$, and $K\cap H\neq\emptyset$ only for the upper half-plane $H$, $\Omega=K\cap H=F_1(\Omega)\cup F_2(\Omega)\cup F_3(K)$ satisfies the BGD condition. The case $n=0,k=1$ is the case that $\Omega=K\setminus\{p_3\}$ and clearly $\Omega=F_1(K)\cup F_2(K)\cup F_3(\Omega)$ satisfies the BGD condition.
Otherwise, assume $n\geq1$ and $k\neq 0$. For the upper half-plane $H_1$, the geometric boundary of $\Omega_1$ is a disjoint union of $F_\omega(I)$ for some $\omega\in\Sigma^n$. Hence $\Omega$ can be written as a union of $n$-cells and $F_\omega(K\setminus I)$ and satisfies the BGD condition. For the lower half-plane $H_2$, the geometric boundary of $\Omega_2$ is a disjoint union of $F_\omega(\{p_1,p_2\})$ for some $\omega\in\Sigma^n$. Similarly, $\Omega$ can be written as a union of $n$-cells and $F_\omega(K\setminus \{p_3\})$ and also satisfies the BGD condition.

\medskip

 Next we assume the slope of $L$ is in $(0,\sqrt{3})$, then $L$ does not include any edge of the triangle $F_{\omega}(T)$ for any $\omega\in \cup_{n\geq0}\Sigma^n$. We only prove the proposition for $\Omega_1$, since the $\Omega_2$ case is similar. In the following, we write $\Omega=\Omega_1$ and $H=H_1$.

 For $k\geq0$ and ${\bf x}=\frac{x_1}{2^k}{\bf e}_1+\frac{x_2}{2^k}{\bf e}_2$ with integers $x_1, x_2$, define a map $\varphi_{k,{\bf x}}: \mathbb R^2\rightarrow \mathbb R^2$ such that
 \begin{equation*}
 \varphi_{k,{\bf x}}(z)=\frac z{2^k}+{\bf x}.
 \end{equation*}
Let  \begin{equation*}
\mathcal C_k=\{ \varphi_{k,{\bf x}}(K):\ {\bf x}=\frac{x_1}{2^k}{\bf e}_1+\frac{x_2}{2^k}{\bf e}_2, x_1, x_2\in \mathbb Z, \varphi_{k,{\bf x}}(K)\cap H\neq\emptyset,\varphi_{k,{\bf x}}(K)\cap L\neq\emptyset \}.
\end{equation*}
For $\alpha=\varphi_{k,{\bf x}}(K)\in \mathcal C_k$, denote $\Omega_{\alpha}=\varphi_{k,{\bf x}}^{-1}(\varphi_{k,{\bf x}}(K)\cap H)$ and $D_{\alpha}=\varphi_{k,{\bf x}}^{-1}(\varphi_{k,{\bf x}}(K)\cap L)$.

Assume $p,q\in V_n$ for some $n\geq0$.

\medskip

{\it Claim. The collection
$
\{(\Omega_\alpha, D_{\alpha}): \alpha\in \mathcal C_n\}
$
is finite and satisfies the BGD condition.}

\medskip

Indeed, noting that $p,q\in\frac{\mathbb Z}{2^n}{\bf e}_1+\frac{\mathbb Z}{2^n}{\bf e}_2$, by periodicity, the collection $\{(\Omega_\alpha, D_{\alpha}): \alpha\in \mathcal C_n\}$ is determined by those $\varphi_{n,{\bf x}}(K)$ with $\varphi_{n,{\bf x}}(K)\cap \overline{pq}\neq\emptyset$ (where $\overline{pq}$ is the line segment connecting $p,q$), hence is finite.

It suffices to check that for any $\beta\in \mathcal C_{n+1}$, $\Omega_\beta=\Omega_\alpha$ for some $\alpha\in \mathcal C_n$. Assume $\alpha=\varphi_{n,{\bf x}}(K)\in \mathcal C_n$ for some ${\bf x}=\frac{x_1}{2^{n}}{\bf e}_1+\frac{x_2}{2^{n}}{\bf e}_2$. Then
\begin{align*}
\Omega_\alpha&=2^{n}\left(\left(\frac{1}{2^{n}}K+{\bf x}\right)\cap H-{\bf x}\right)\\
&=(K+2^{n}{\bf x})\cap 2^{n}H-2^{n}{\bf x}\\
&=K\cap 2^{n}(H-{\bf x})\\
&=K\cap (2^{n}H-x_1{\bf e}_1-x_2{\bf e}_2).
\end{align*}
Similarly, for $\beta=\varphi_{n+1,{\bf y}}(K)\in \mathcal C_{n+1}$  with ${\bf y}=\frac{y_1}{2^{n+1}}{\bf e}_1+\frac{y_2}{2^{n+1}}{\bf e}_2$, we have
\begin{align*}
\Omega_\beta=K\cap (2^{n+1}H-y_1{\bf e}_1-y_2{\bf e}_2).
\end{align*}
Since $H$ is determined by the line $L=\{z\in\mathbb R^2:\ \overrightarrow{Oz}=t\overrightarrow{pq}+\overrightarrow{Op},\ t\in\mathbb R\}$, we have $2^{k}H$ is determined by the line $\{z\in \mathbb R^2:\ \overrightarrow{Oz}=t\overrightarrow{pq}+2^{k}\overrightarrow{Op},\ t\in\mathbb R\}$, for any $k\geq 0$. Now for $k=n$ or $k=n+1$, $2^{k}\overrightarrow{Op}=k_1{\bf e}_1+ k_2{\bf e}_2$ for some integers $k_1,k_2$ since $p\in V_n$. Then the half-plane $2^{n+1}H-y_1{\bf e}_1-y_2{\bf e}_2$ is determined by some line $L'=\{z\in\mathbb R^2:\ \overrightarrow{Oz}=t\overrightarrow{pq}+k'_1{\bf e}_1+k'_2{\bf e}_2,\ t\in\mathbb R\}$, for some integers $k_1'$, $k_2'$. Hence we see that for any $\beta\in \mathcal C_{n+1}$, there is some $\alpha\in \mathcal C_n$ such that $\Omega_\beta=\Omega_\alpha$. The claim holds.

\medskip

From the claim, we immediately see that $\{\Omega_\alpha:\ \alpha\in \mathcal C_k,\ 0\leq k\leq n\}$ satisfies the BGD condition. In particular, $\Omega\in\{\Omega_\alpha:\ \alpha\in\mathcal C_0\}$ also satisfies the BGD condition.
\end{proof}

\medskip

Now we illustrate two particular situations and calculate their flux transfer matrices.

\medskip

 {{\bf 1.} $p=p_1,q=p_2,$ $\Omega=K\setminus\overline{p_1p_2}$, $D=\overline{p_1p_2}$} (see Figure \ref{figure2}).
 \begin{figure}[h]
	\includegraphics[width=5cm]{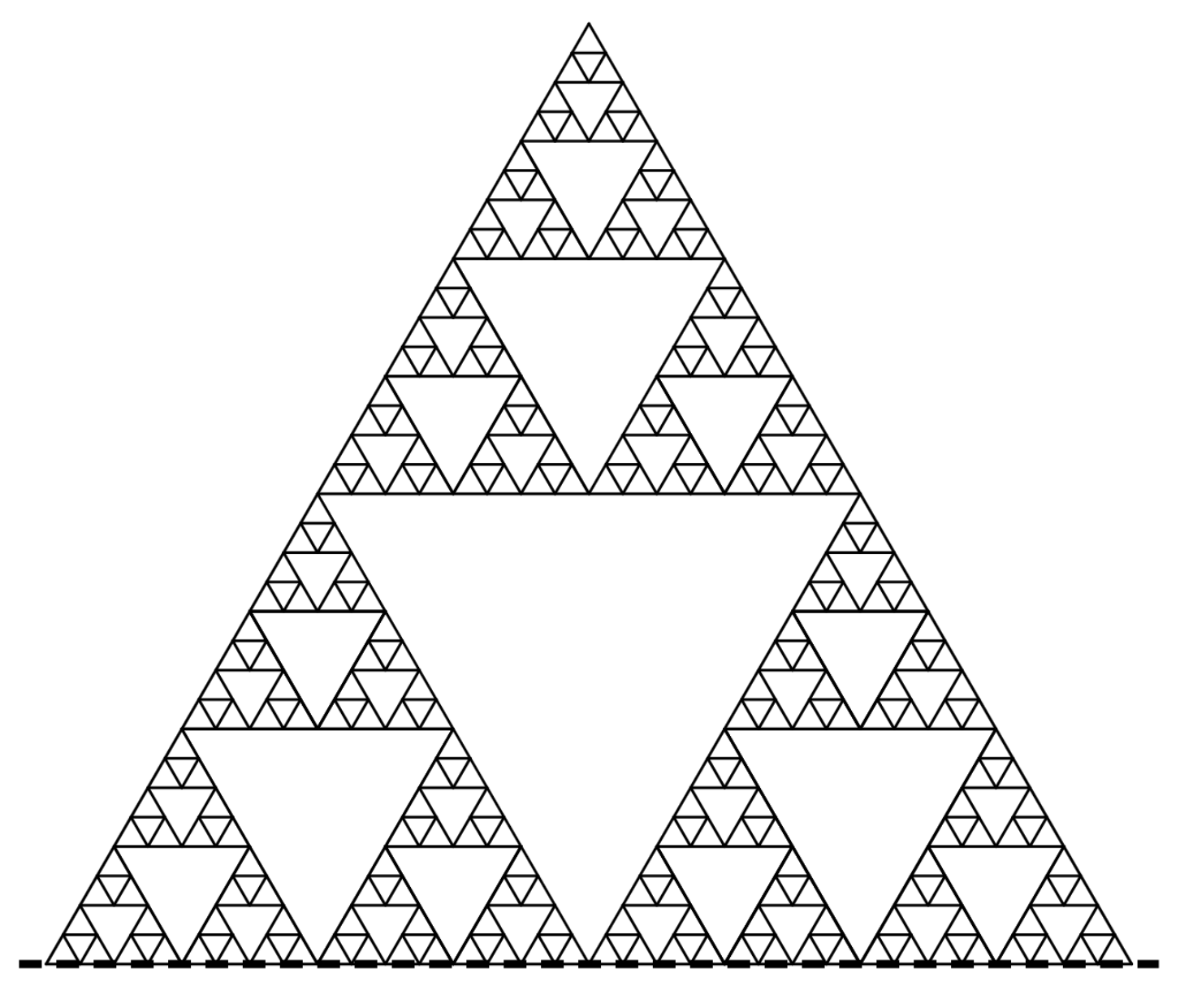}
	\begin{picture}(0,0)
	\put(-5,0){$p_2$}
	\put(-152,0){$p_1$}
	\put(-79,123){$p_3$}
	\end{picture}	
	\caption{a domain in SG} \label{figure2}
\end{figure}
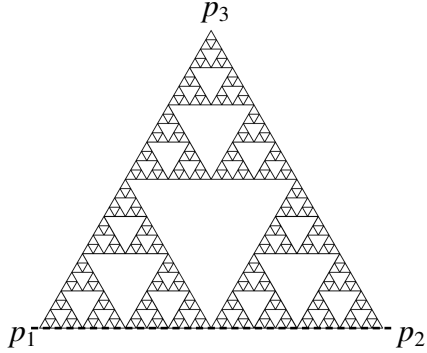
 This is an example in \cite{OS} by Owen and Strichartz. By using a Haar basis expansion method, they proved that for this domain $\Omega$, the hitting probability from $p_3$ to $\partial\Omega$ is the normalized uniform measure on $\partial\Omega$. We refer to \cite{GKQS,CQ,KT22} for further discussions. Under the general setting here, we can reformulate their result as follows. The boundary $D=\overline{p_1p_2}$ can be viewed as a self-similar set generated by the IFS $\{F_1,F_2\}$, and this example satisfies the BGD condition, which has the directed graph $(\mathcal{A},\Gamma)$ with only one vertex $\mathcal{A}=\{1\}$ and two edges $\Gamma=\{\gamma_1,\gamma_2\}$, {each of which is} from $1$ to  itself, where $\gamma_i$ is associated with the contraction map $F_i$ for $i=1,2$ respectively. By using that the renormalizing factor $r=\frac35$ together with the self-similarity, it is not hard to compute the effective resistance $R_{\Omega}(\partial\Omega,p_3)=\frac37$, and the unit flow from $\partial\Omega$ to $p_3$ flows outwards $F_i(\Omega)$ through $F_i(p_3)$ with flux $\frac12$ for $i=1,2$. Thus the flux transfer matrices associated with $\gamma_1$ and $\gamma_2$ are $$M_{\gamma_1}= M_{\gamma_2}=\left(
                                                                                                                       \begin{array}{ccc}
                                                                                                                         0&0&0\\
                                                                                                                         0&0&0\\
                                                                                                                         0&0&\frac12\\
                                                                                                                       \end{array}
                                                                                                                     \right).
$$ Then by Theorem \ref{th3.1}, we can compute by using the product of $M_{\gamma_1},M_{\gamma_2}$ to obtain that the hitting probability from $p_3$ is the $(\frac12,\frac12)$-self-similar measure on $\partial\Omega$.

\medskip

{\bf 2.} $p=p_3$, $q=(\frac12,0)$, $\Omega=\{x=(x_1,x_2)\in K:\ x_1<\frac12\}$.  Note that $D=L\cap K$ consists of countably many points. By solving systems of countably infinite linear equations, Li and Strichartz \cite{LS} computed explicitly the hitting probability from $p_1$ to $\partial\Omega$ (homeomorphic to $D$). See also \cite{CQ} for generalizations by Cao and the second author.

Write $\Omega_1=\Omega$ with boundary $D_1=D$, and $\Omega_2=K\setminus\{p_2\}$ with boundary $D_2=\{p_2\}$, see Figure \ref{figure3}.
 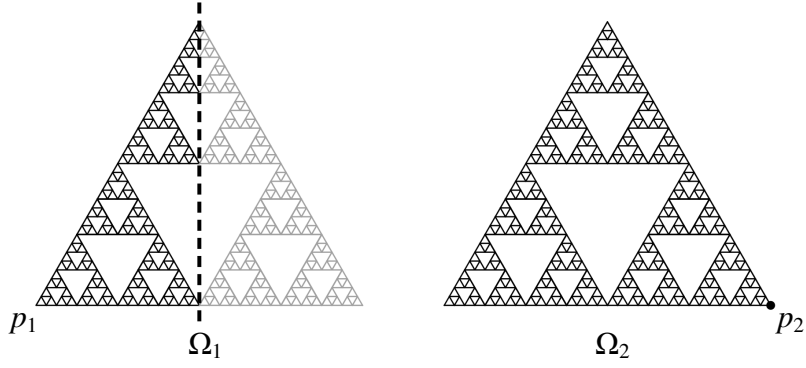
\begin{figure}[h]
	\includegraphics[width=10cm]{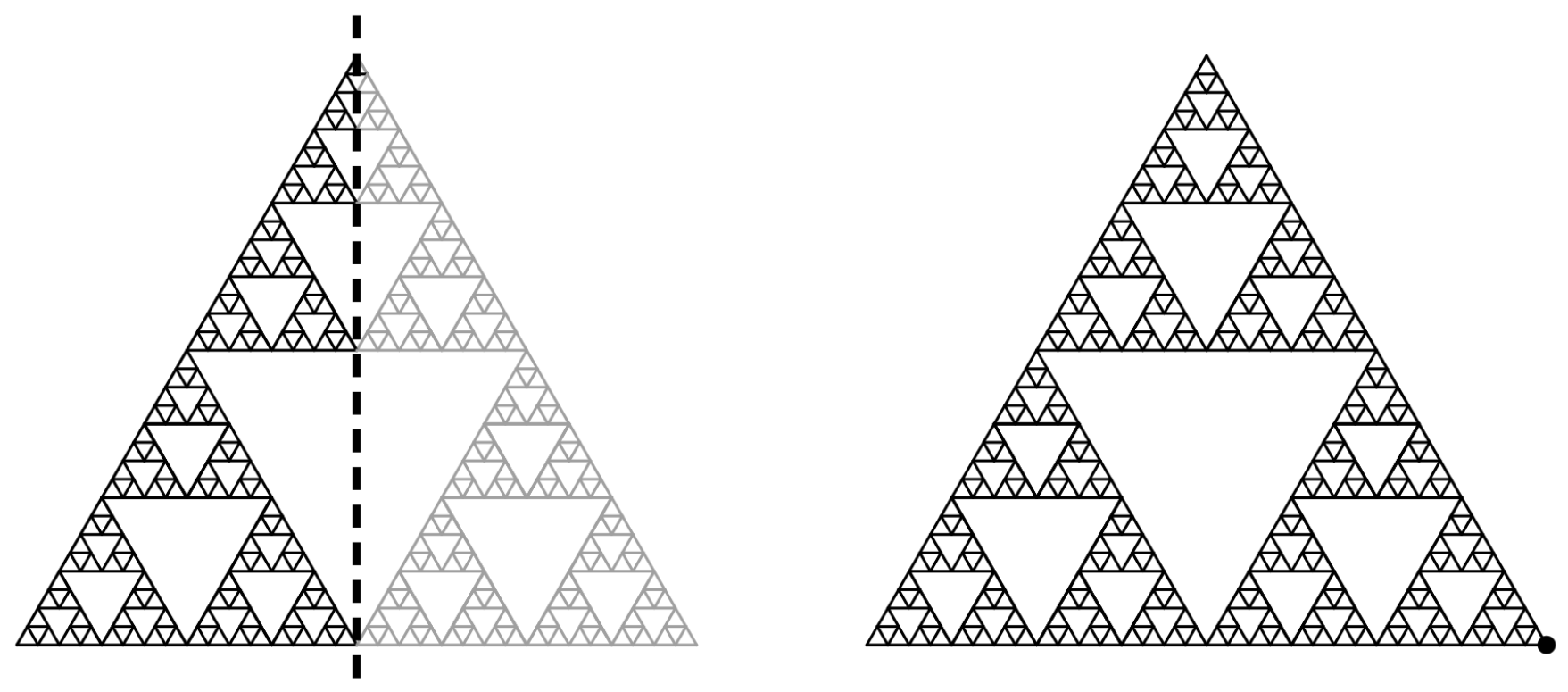}
	\begin{picture}(0,0)
	\put(-5,0){$p_2$}\put(-295,0){$p_1$}
	\put(-227,-10){$\Omega_1$}\put(-73,-10){$\Omega_2$}
	\end{picture}	
	\caption{domains in SG} \label{figure3}
\end{figure}
Then $\{(\Omega_i,D_i)\}_{i=1}^2$ satisfies the BGD condition with a directed graph $(\mathcal{A},\Gamma)$: $\mathcal{A}=\{1,2\}$, $\Gamma=\{\gamma_1,\gamma_2,\gamma_3\}$, where $\gamma_1$ is from $1$ to $1$ associated with $F_3$, $\gamma_2$ is from $1$ to $2$ associated with $F_1$, and $\gamma_3$ is from $2$ to $2$ associated with $F_2$; see Figure \ref{figure4}.
 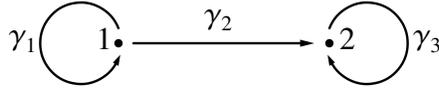
\begin{figure}[h]
	\includegraphics[width=5cm]{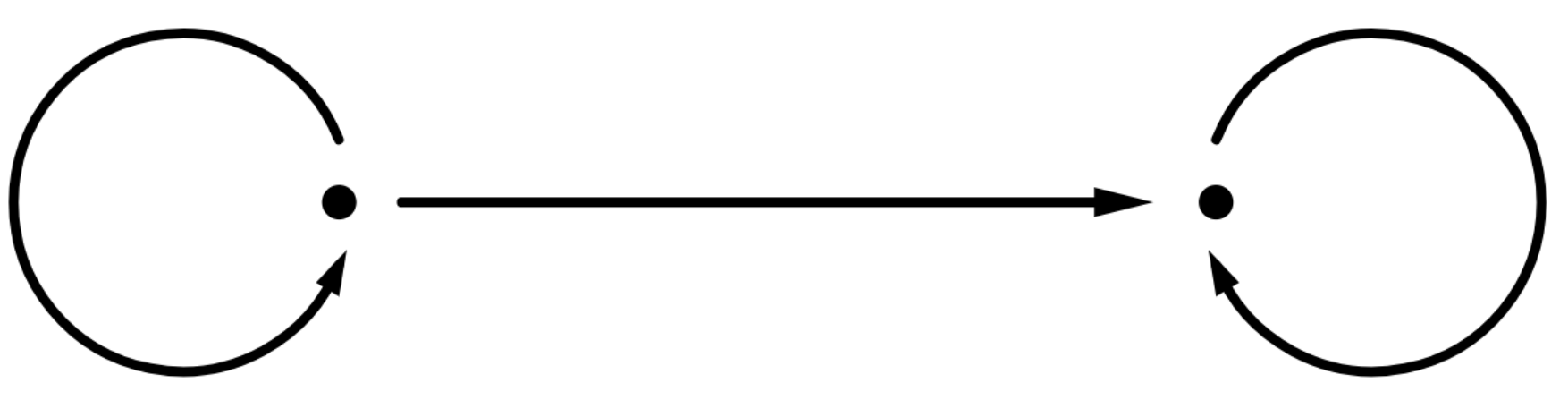}
	\begin{picture}(0,0)
	
	\put(-123,17){$1$}\put(-31,17){$2$}\put(-3,19){$\gamma_3$}\put(-156,19){$\gamma_1$}\put(-82,27){$\gamma_2$}
	\end{picture}	
	\caption{The directed graph $(\mathcal{A},\Gamma)$} \label{figure4}
\end{figure}

Then by a direct calculation, the associated flux transfer matrices are
\begin{equation*}
M_{\gamma_1}=\left(
                                                                                                                       \begin{array}{ccc}
                                                                                                                         1/3 &0&0\\
                                                                                                                         0 &0&0\\
                                                                                                                         0 &0&0\\
                                                                                                                       \end{array}
                                                                                                                     \right), \ M_{\gamma_2}=\left(
          \begin{array}{ccc}
            1 &0& -1/3\\
             0 &0& 0\\
              0&0& 0\\
          \end{array}
        \right), M_{\gamma_3}=\left(
             \begin{array}{ccc}
               2/3 &0& 1/3 \\
                0 &0& 0\\
               1/3 &0& 2/3 \\
             \end{array}
           \right).
\end{equation*}
            By Theorem \ref{th3.1}, we see that the hitting probability from $p_1$ to the boundary $\partial\Omega$ is
\begin{equation*}
\sum_{n=0}^\infty\frac{2}{3^{n+1}}\delta_{F_{3^n1}(p_2)},
\end{equation*}
 where $\delta_x$ is the Dirac measure at $x$.

\medskip

\subsection{Example: hexagasket}\label{ex2}
The hexagasket is a p.c.f. self-similar set generated by the IFS $\{F_i\}_{i=1}^6$, with $F_i(x)=\frac13(x-p_i)+p_i$, where $V_0=\{p_i\}_{i=1}^6$ are the six vertices of a regular hexagon in $\mathbb R^2$. The standard resistance form $(\mathcal E,\mathcal F)$ on $K$ satisfies \cite{S}
\begin{equation*}
\mathcal E[u]=\frac{7}{3}\sum_{i=1}^6\mathcal E[u\circ F_i], \quad \forall u\in\mathcal{F}.
\end{equation*}
Set $p_1=(-1,0)$, $p_2=(-\frac12,-\frac{\sqrt3}2)$, $p_3=(\frac12,-\frac{\sqrt3}2)$, $p_4=(1,0)$, $p_5=(\frac12,\frac{\sqrt3}2)$, $p_6=(-\frac12,\frac{\sqrt3}2)$. Let $D=\overline{p_1p_4}\cap K$, which is a middle-third Cantor set. Let $H=\{x=(x_1,x_2):\ x_2>0\}$ be the (open) upper half-plane. We define the domain $\Omega=K\cap H$, with boundary $D$; see Figure \ref{figurehex}. Then $(\Omega,D)$ satisfies the BGD condition with the directed graph $(\mathcal{A},\Gamma)$ given by $\mathcal{A}=\{1\}$ and $\Gamma=\{\gamma_i\}_{i=1}^2$, where both $\gamma_1$ and $\gamma_2$ are from $1$ to itself.
\begin{figure}[h]
	\includegraphics[width=8cm]{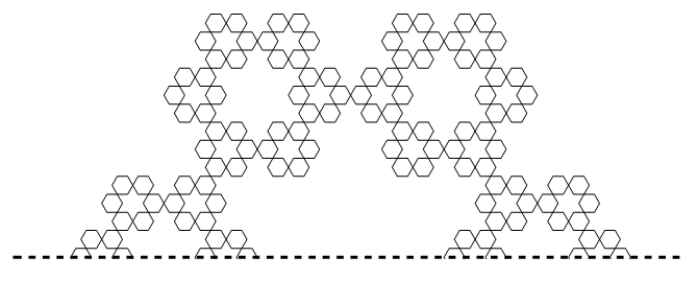}
	\begin{picture}(0,0)
\put(-219,20){$p_1$}\put(-22,20){$p_4$}\put(-69,98){$p_5$}\put(-174,98){$p_6$}
	\end{picture}	
	\caption{a half domain in the hexagasket} \label{figurehex}
\end{figure}
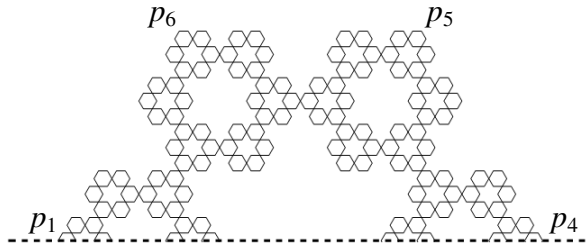
The associated flux transfer matrices are
\begin{align*}
&M_{\gamma_1}=\left(
       \begin{array}{cccccc}
         0 & 0&0&0&0&0 \\
         0 & 0&0&0&0&0 \\
       0 & 0&0&0&0&0 \\
       0 & 0&0&0&0&0\\
       0 & 0&0&0&1/3&0\\
       0 & 0&0&0&2/3&0\\
       \end{array}
     \right), M_{\gamma_2}=\left(
          \begin{array}{cccccc}
            0 & 0&0&0&0&0 \\
            0 & 0&0&0&0&0 \\
       0 & 0&0&0&0&0 \\
       0 & 0&0&0&0&0\\
       0 & 0&0&0&0&2/3\\
       0 & 0&0&0&0&1/3\\
       \end{array}
     \right).
\end{align*}
By Theorem \ref{th3.1}, we obtain the hitting probability from $p_5$ (or $p_6$) is a twisted $(1/3,2/3)$-self-similar measure on $\partial\Omega$.

\subsection{Example: Vicsek set}\label{ex3}
The Vicsek set is a p.c.f. self-similar set generated by the IFS $\{F_i\}_{i=1}^5$, with $F_i(x)=\frac13(x-p_i)+p_i$, where $V_0=\{p_i\}_{i=1}^4$ are the four corner vertices of a square and $p_5$ is its center.  The standard resistance form $(\mathcal E,\mathcal F)$ on $K$  satisfies \cite{S}
\begin{equation*}
\mathcal E[u]=3\sum_{i=1}^5\mathcal E[u\circ F_i], \quad \forall u\in\mathcal{F}.
\end{equation*}
Let $D_1=\overline{p_1p_2}\cap K$ and $D_2=(\overline{p_1p_2}\bigcup\overline{p_2p_3})\cap K$. Then $D_1$ is a middle-third Cantor set and $D_2$ is a union of two copies of $D_1$. Let $\Omega_1=K\setminus D_1$, $\Omega_2=K\setminus D_2$ with boundaries $D_1$, $D_2$ respectively; see Figure \ref{figure5}.
 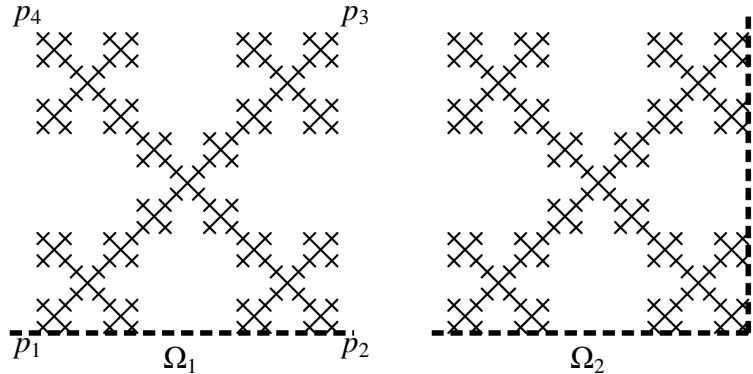
\begin{figure}[h]
	\includegraphics[width=10cm]{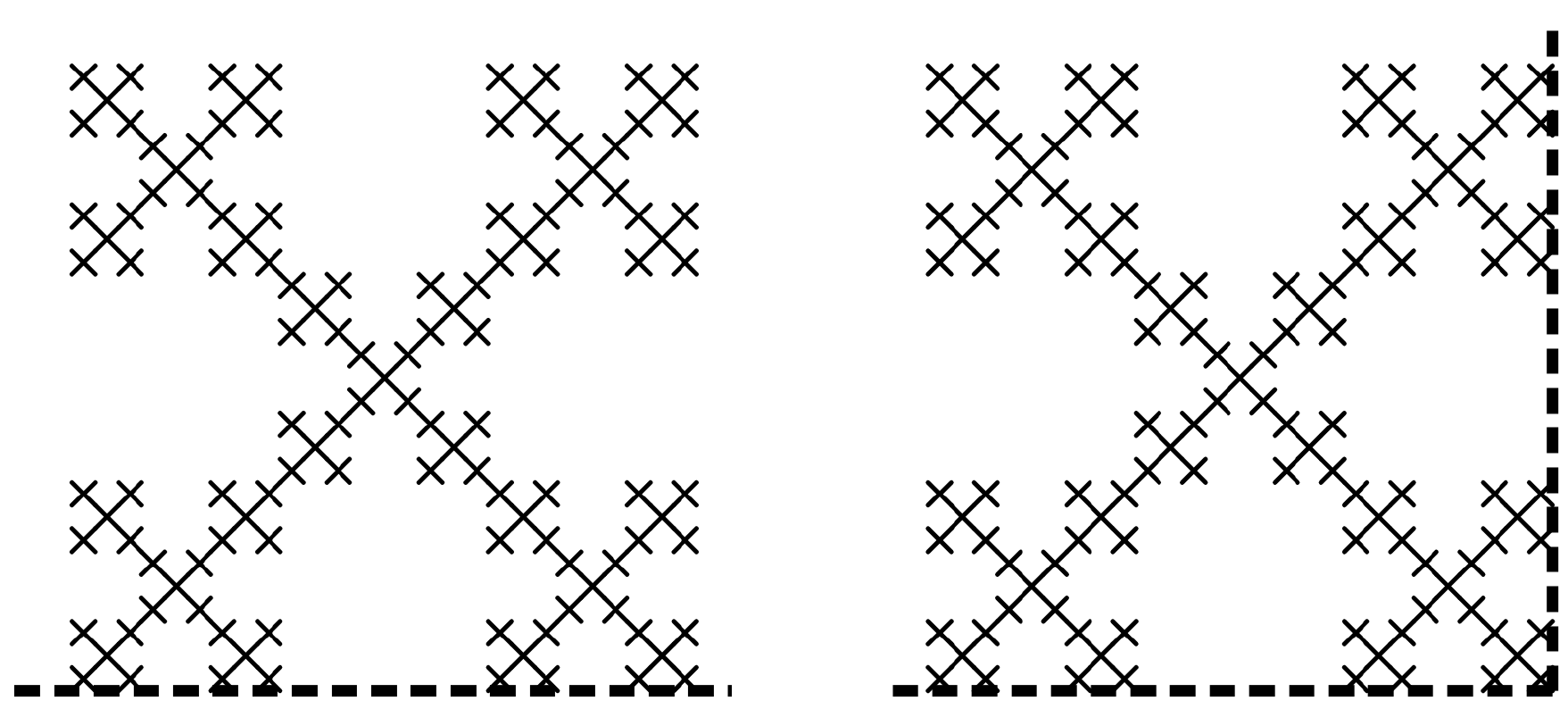}
	\begin{picture}(0,0)
	\put(-160,122){$p_3$}\put(-284,122){$p_4$}
	\put(-284,-4){$p_1$}\put(-160,-4){$p_2$}	\put(-227,-10){$\Omega_1$}\put(-73,-10){$\Omega_2$}
	\end{picture}	
	\caption{domains in the Vicsek set with Cantor boundaries} \label{figure5}
\end{figure}
Then $\{(\Omega_i,D_i)\}_{i=1}^2$ satisfies the BGD condition with the directed graph $(\mathcal{A},\Gamma)$ given by $\mathcal{A}=\{1,2\}$ and $\Gamma=\{\gamma_i\}_{i=1}^5$ as illustrated in Figure \ref{figure6}, where for brevity we treat domains modulo symmetry.
 \begin{figure}[h]
	\includegraphics[width=5cm]{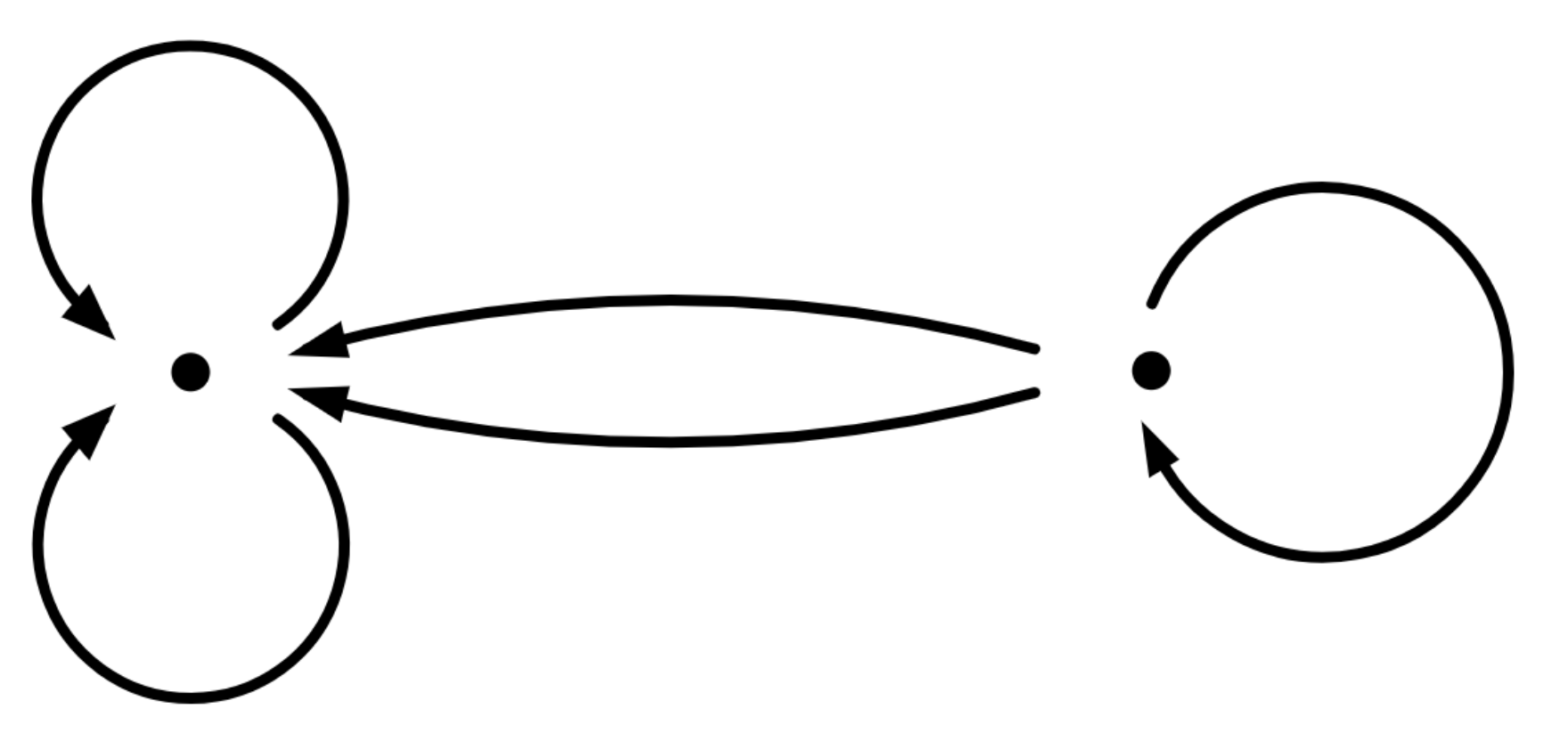}
	\begin{picture}(0,0)
\put(-125,29){$1$}\put(-38,29){$2$}\put(-4,30){$\gamma_5$}\put(-154,15){$\gamma_1$}\put(-154,46){$\gamma_2$}\put(-85,18){$\gamma_3$}\put(-85,46){$\gamma_4$}
	\end{picture}	
	\caption{The directed graph $(\mathcal{A},\Gamma)$ in Example \ref{ex3}} \label{figure6}
\end{figure}
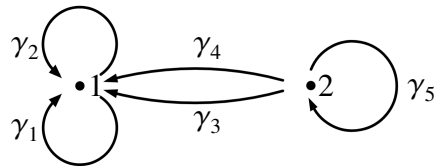
The associated contraction maps of $\{\gamma_i\}_{i=1}^5$ are $F_1, F_2, F_1$, $F_3\circ \kappa$ and $F_2$, where $\kappa$ is the counterclockwise {rotation} by $\frac\pi 2$ around the center $p_5$. By a direct computation, we obtain the associated flux transfer matrices are
\begin{align*}
&M_{\gamma_1}=\left(
       \begin{array}{cccc}
       0 & 0&0&0 \\
       0 & 0&0&0 \\
        0& 0 & 1/2 &0\\
        0& 0 & 1/2 &0\\
       \end{array}
     \right), M_{\gamma_2}=\left(
          \begin{array}{cccc}
           0 & 0&0&0 \\
       0 & 0&0&0 \\
         0&0 & 0 & 1/2 \\
           0&0 & 0 & 1/2 \\
          \end{array}
        \right), M_{\gamma_3}=\left(
          \begin{array}{cccc}
           0 & 0&0&0 \\
       0 & 0&0&0 \\
         0&0 & 0 & 0 \\
           0&0 & \frac{\sqrt{69}-7}{4} & 0 \\
             \end{array}
           \right),\\ &M_{\gamma_4}=\left(
                \begin{array}{cccc}
           0 & 0&0&0 \\
       0 & 0&0&0 \\
         0&0 & 0 & 0 \\
           0&0 & 0 & \frac{\sqrt{69}-7}{4}\\
             \end{array}
                         \right), M_{\gamma_5}=\left(
               \begin{array}{cccc}
           0 & 0&0&0 \\
       0 & 0&0&0 \\
         0&0 & 0 & 0 \\
           0&0 & 0 & \frac{9-\sqrt{69}}{2} \\
             \end{array}
           \right).
\end{align*}
 For $\Omega_1$, the hitting probability from $p_3$ (or $p_4$) is the $(1/2,1/2)$-self-similar measure on $\partial\Omega_1$. For $\Omega_2$, the hitting probability $\mu$ from $p_4$ to $\partial\Omega_2$ can be described as: for any $k\geq0$, the restriction of $\mu$ on the boundary of $F_{2^k1}(\Omega_1)$
 is a $(1/2,1/2)$-self-similar measure with total weight $\left(\frac{9-\sqrt{69}}{2}\right)^k(\frac{\sqrt{69}-7}{4})$.

\bibliographystyle{siam}

\bigskip
\end{document}